\documentclass[12pt]{amsart}
\usepackage[margin=2.55cm]{geometry}
\usepackage{amscd}      
\usepackage{amssymb}
\usepackage{hyperref}
\usepackage{amsmath, amsthm, graphics,epic,eepic}
\usepackage{cite}
\usepackage{xypic}     
\LaTeXdiagrams    
\usepackage[all]{xy}
\xyoption{2cell} \UseAllTwocells \xyoption{frame} \CompileMatrices
\usepackage{latexsym}
\usepackage{epsfig}
\usepackage{amsfonts}
\usepackage{enumerate}
\usepackage{times}

\allowdisplaybreaks[3]
\newtheorem{prop}{Proposition}
\newtheorem{lem}[prop]{Lemma}
\newtheorem{cor}[prop]{Corollary}
\newtheorem{thm}[prop]{Theorem}
\theoremstyle{definition}
\newtheorem{defn}[prop]{Definition}
\newtheorem{rmk}[prop]{Remark}
\newtheorem{notation}[prop]{Notation}

\newcommand{\bpsi}{\bar{\psi}}

\newcommand{\tQ}{\widetilde{Q}}

\newcommand{\bSigma}{\mathbf{\Sigma}} 

\newcommand{\cE}{\mathcal{E}}
\newcommand{\cV}{\mathcal{V}}

\newcommand{\X}{\mathcal{X}}
\newcommand{\Y}{\mathcal{Y}}

\newcommand{\sL}{\mathcal{L}}

\newcommand{\bc}{\mathbf{c}}
\newcommand{\be}{\mathbf{e}}

\newcommand{\sH}{\mathcal{H}}
\newcommand{\cO}{\mathcal{O}}

\newcommand{\fun}{\mathbf{1}}

\def\<{\left\langle}
\def\>{\right\rangle}

\DeclareMathOperator{\sech}{sech}

\DeclareMathOperator{\Ch}{CH}
\let\Box=\relax
\DeclareMathOperator{\Box}{Box}

\DeclareMathOperator{\ord}{ord}
\DeclareMathOperator{\NE}{NE}
\DeclareMathOperator{\Pic}{Pic}

\DeclareMathOperator{\Cone}{Cone} 
\DeclareMathOperator{\Sym}{Sym} 
 
\DeclareMathOperator{\age}{age} 
 
\DeclareMathOperator{\ch}{ch}

\newcommand{\cC}{\mathcal{C}}

\newcommand{\CR}{\text{\rm CR}}

\newcommand{\tor}{\text{\rm tor}}
\newcommand{\tw}{\text{\rm tw}}
\newcommand{\non}{\text{\rm non}}

\newcommand{\Cstar}{\mathbb{C}^\times}

\newcommand{\CC}{\mathbb{C}}
\newcommand{\PP}{\mathbb{P}}
\newcommand{\QQ}{\mathbb{Q}}
\newcommand{\RR}{\mathbb{R}}
\newcommand{\ZZ}{\mathbb{Z}}
\newcommand{\TT}{\mathbb{T}}
\newcommand{\LL}{\mathbb{L}}
\newcommand{\NN}{\mathbb{N}}

\newcommand{\Novikov}{\boldsymbol{\Lambda}} 

\begin{document}

\title{Some Applications of the Mirror Theorem for Toric Stacks}
\author[Coates]{Tom Coates}
\author[Corti]{Alessio Corti}
\address{Department of Mathematics\\ Imperial College London\\ 
180 Queen's Gate\\ London SW7 2AZ\\ United Kingdom}
\email{t.coates@imperial.ac.uk}
\email{a.corti@imperial.ac.uk}

\author[Iritani]{Hiroshi Iritani}
\address{Department of Mathematics\\ Graduate School of Science\\ 
Kyoto University\\ Kitashirakawa-Oiwake-cho\\ Sakyo-ku\\ 
Kyoto\\ 606-8502\\ Japan}
\email{iritani@math.kyoto-u.ac.jp}

\author[Tseng]{Hsian-Hua Tseng}
\address{Department of Mathematics\\ Ohio State University\\ 
100 Math Tower, 231 West 18th Ave. \\ Columbus \\ OH 43210\\ USA}
\email{hhtseng@math.ohio-state.edu}

\date{\today}

\begin{abstract}
  We use the mirror theorem for toric Deligne--Mumford stacks, proved
  recently by the authors and by Cheong--Ciocan-Fontanine--Kim, to
  compute genus-zero Gromov--Witten invariants of a number of toric
  orbifolds and gerbes.  We prove a mirror theorem for a class of
  complete intersections in toric Deligne--Mumford stacks, and use
  this to compute genus-zero Gromov--Witten invariants of an orbifold
  hypersurface.
\end{abstract}

\maketitle

\section{Introduction}

Given a symplectic orbifold or Deligne--Mumford stack $\X$, one might
want to calculate the Gromov--Witten invariants of $\X$:
\[
\<
a_1\bpsi^{k_1},\ldots,a_n\bpsi^{k_n}
\>_{g,n,d}^{\X}
\]
where $a_1,\ldots,a_n$ are classes in the Chen--Ruan orbifold
cohomology of $\X$ and $k_1,\ldots,k_n$ are non-negative integers.
Gromov--Witten invariants carry information about the enumerative
geometry of $\X$: roughly speaking they count the number of orbifold
curves in $\X$, of genus $g$ and degree $d$, that pass through certain
cycles (recorded by the classes $a_i$) and satisfy certain constraints
on their complex structure.  Computing Gromov--Witten invariants is in
general hard, but one can often compute genus-zero Gromov--Witten
invariants using mirror symmetry.  A mirror theorem for toric
Deligne--Mumford stacks was proved recently by the authors
\cite{ccit2} and, independently, by Cheong--Ciocan-Fontanine--Kim
\cite{ccfk}.  In what follows we give various applications of this
mirror theorem.  We compute genus-zero Gromov--Witten invariants of a
number of toric Deligne--Mumford stacks; prove a mirror theorem
(Theorem \ref{thm:ci_mirror}) 
for certain complete intersections in
toric Deligne--Mumford stacks; and use this to compute genus-zero
Gromov--Witten invariants of an orbifold hypersurface.   Along the way
we make a technical point that may be useful elsewhere: showing that
one can apply Coates--Givental/Tseng-style hypergeometric
modifications to $I$-functions, rather than just to $J$-functions
(Theorem~\ref{thm:twisted_I_is_on_the_twisted_cone}).  The mirror theorem proved here (Theorem~\ref{thm:ci_mirror}) is new, and contains all known mirror theorems for toric complete intersections as special cases (see e.g.~\cite{ccit,Givental:toric}).  It plays a key role in the recent proof of the Crepant Transformation Conjecture for  complete intersections in toric Deligne--Mumford stacks~\cite{CIJ}. 

This paper is written with two purposes in mind.  It provides a
reasonably self-contained guide that should help the reader to apply
our mirror theorems to new examples.  It also increases the number of
explicit, non-trivial calculations of orbifold Gromov--Witten
invariants in the literature.  Orbifold Gromov--Witten theory is
fraught with technical subtleties, and we hope that our calculations
will be useful for others, as test examples for more sophisticated
theories.  The examples also demonstrate a practical advantage of our
mirror theorems over existing methods \cite{cclt,Woodward:1,Woodward:2,Woodward:3,gw1}, in
that they often allow the direct determination of genus-zero
Gromov--Witten invariants with insertions from twisted sectors,
without needing to resort to the WDVV equation or reconstruction
theorems \cite{Iritani:QDM,Rose}.

Let $\X$ be an algebraic Deligne--Mumford stack equipped with the
action of a (possibly-trivial) torus $\TT$.  Suppose that $\X$ is
sufficiently nice that one can define $\TT$-equivariant Gromov--Witten
invariants; this is the case, for example, if $\X$ is smooth as a
stack and the coarse moduli space $X$ of $\X$ is semi-projective
(projective over affine).  Let $H^\bullet_{\CR, \TT}(\X)$ denote the
$\TT$-equivariant Chen--Ruan cohomology of $\X$ (see \S\ref{sec:CR}).
Let $\Novikov(R)$ denote the Novikov ring of $\X$; this is a completion of
the group ring $R[H_2(\X;\ZZ) \cap \NE(\X)]$ of the semigroup 
$H_2(\X;\ZZ) \cap \NE(\X)$ 
generated by classes of effective curves. 
  Following Givental \cite{gi},
Tseng has defined a symplectic structure on:
\[
\sH := H^\bullet_{\CR, \TT}(\X) \otimes_{H_{\TT}^\bullet({\rm pt})} 
\Novikov\left(H_{\TT}^\bullet({\rm pt})(\!(z^{-1})\!)\right)
\]
and a Lagrangian submanifold $\sL$ of $\sH$ that encodes all
genus-zero Gromov--Witten invariants of $\X$ \cite{ts}.  We will not
give a precise definition of $\sL$ in this paper, referring the reader
to \cite[\S2]{ccit2} for a detailed discussion.  For us, what will be
important is that $\sL$ determines and is determined by Givental's
$J$-function:
\begin{equation}
  \label{eq:J_function}
  J_\X(t,z) = z + t + \sum_{d \in H_2(X;\ZZ)} \sum_{n=0}^\infty
  \sum_{k=0}^\infty \sum_\alpha
  \frac{Q^d}{n!}
  \<t,t,\ldots,t,\phi_\alpha \psi^k\>^\X_{0,n+1,d}
  \phi^\alpha z^{-k-1}
\end{equation}
where $t \in H^\bullet_{\CR, \TT}(\X)$; $z$ is a formal variable;
$Q^d$ is the representative of $d$ in the Novikov ring~$\Novikov$; the
correlator denotes a Gromov--Witten invariant, exactly as in
\cite[\S2]{ccit2}; and $\{\phi_\alpha\}$, $\{\phi^\alpha\}$ denote
bases for $H^\bullet_{\CR, \TT}(\X)$ which are dual with respect to
the pairing on Chen--Ruan cohomology.  The submanifold $\sL$
determines the $J$-function because $J_\X(t,-z)$ is the unique point
on $\sL$ of the form ${-z} + t + O(z^{-1})$, where $O(z^{-1})$ is a
power series in $z^{-1}$.  The $J$-function determines $\sL$ because
it determines all genus-zero Gromov--Witten invariants of $\X$ with
descendant insertions at one or fewer marked points; it thus
determines all genus-zero invariants with descendant insertions at two
marked points via \cite[Proposition~2.1]{GT}, and determines all other
genus-zero invariants via the Topological Recursion Relations
\cite[\S2.5.7]{ts}. To determine the genus-zero Gromov--Witten
invariants of $\X$, therefore, it suffices to determine the
$J$-function $J_\X(t,z)$.  In \S\ref{sec:examples} and
\S\ref{sec:ci_examples} below we use mirror theorems to determine the
$J$-function of a number of Deligne--Mumford stacks $\X$.

The reader may be interested in the quantum orbifold cohomology ring
of $\X$. Recovering quantum cohomology from the $J$-function is
straightforward: general theory implies that $J_\X(t,z)$ satisfies a
system of differential equations:
\[
z
\frac{\partial}{\partial t^\alpha}
\frac{\partial}{\partial t^\beta}
J_\X(t,z) = 
\sum_\gamma {c_{\alpha \beta}}^\gamma(t) 
\frac{\partial}{\partial t^\gamma}
J_\X(t,z)
\]
where $t = \sum_\alpha t^\alpha \phi_\alpha$ and the coefficients
${c_{\alpha \beta}}^\gamma(t)$ are the structure constants of the
orbifold quantum product \cite{CR2,Barannikov:quantum,gi2}.  Thus:
\[
z
\frac{\partial}{\partial t^\alpha}
\frac{\partial}{\partial t^\beta}
J_\X(t,z) =
\phi_\alpha \star_t \phi_\beta + O(z^{-1})
\]
where $\star_t$ denotes the big orbifold quantum product with
parameter $t \in H^\bullet_{\CR,\TT}(\X)$.

\subsection*{Acknowledgments}

We thank Ionut Ciocan-Fontanine for useful discussions.  T.C.~was
supported in part by a Royal Society University Research Fellowship,
ERC Starting Investigator Grant number~240123, and the Leverhulme
Trust.  A.C.~was supported in part by EPSRC grants EP/E022162/1 and
EP/I008128/1.  H.I.~was supported in part by EPSRC grant EP/E022162/1
and JSPS Grant-in-Aid for Scientific Research (C)~25400069.
H.-H.T.~was supported in part by a Simons Foundation Collaboration
Grant.

\section{The Mirror Theorem for Toric Deligne--Mumford Stacks}

We assume that the reader is familiar with toric Deligne--Mumford
stacks.  A quick summary of the relevant material can be found in
\cite[\S3]{ccit2}; the theory is developed in detail in \cite{BCS,
  Iwanari1, Iwanari2, FMN, Jiang}.

\subsection{Stacky Fans}

A toric Deligne--Mumford stack is defined by a \emph{stacky fan}
$\bSigma=(N,\Sigma,\rho)$, where $N$ is a finitely generated abelian
group, $\Sigma\subset N_\QQ=N\otimes_{\ZZ}\QQ$ is a rational
simplicial fan, and $\rho \colon \ZZ^{n} \to N$ is a homomorphism with
finite cokernel such that the images of the standard basis vectors in
$\ZZ^n$ under the composition 
$\begin{CD} \ZZ^n @>{\rho}>> N @>>>N_\QQ
\end{CD}$ generate the $1$-dimensional cones of $\Sigma$.  Let
$\LL\subset \ZZ^n$ be the kernel of $\rho$. The exact sequence
\[
\begin{CD}
0 @>>> \LL @>>> \ZZ^n @>{\rho}>> N 
\end{CD} 
\]
is called the {\em fan sequence}.  Let $\rho_i \in N$ denote the image
under $\rho$ of the $i$th standard basis vector in $\ZZ^n$. Let
${\rho}^{\vee} \colon (\ZZ^*)^{n}\to \LL^{\vee} 
:= H^1(\Cone(\rho)^*)$ be the Gale dual
\cite{BCS} of $\rho$. There is an exact sequence
\begin{equation*} 
\begin{CD}
0@>>> N^*@>>>
(\ZZ^*)^n @>{\rho^\vee}>> \LL^\vee
\end{CD} 
\end{equation*}
called the {\em divisor sequence}.  The toric Deligne--Mumford stack
associated to the stacky fan $\bSigma$ admits a canonical action of
the torus $\TT := N \otimes \Cstar$.

\subsection{Chen--Ruan Cohomology}
\label{sec:CR}
Let $\X$ denote the toric Deligne--Mumford stack defined by the stacky
fan $\bSigma = (N,\Sigma,\rho)$.  Let $N_{\tor}$ denote the torsion
subgroup of $N$, let $\overline{N} := N/N_\tor$, and let
$\overline{c}\in \overline{N}$ denote the image of $c \in N$ under the
canonical projection $N\to \overline{N}$.  The \emph{box} of $\bSigma$
is:
\[
\Box(\bSigma):=\left\{b\in N : \text{$\exists \sigma \in \Sigma$ such
    that $\bar{b}=\sum_{i : \bar{\rho}_i \in \sigma}
  a_i\bar{\rho}_i$ for some $a_i$ with $0\leq a_i< 1$} \right\}
\]
Components of the inertia stack $I \X$ are indexed by elements of
$\Box(\bSigma)$, and we write $I\X_b$ for the component of inertia
corresponding to $b\in \Box$.

The $\TT$-equivariant Chen--Ruan
orbifold cohomology \cite{CR1,ccliu} of $\X$ is:
\[
H^\bullet_{\CR, \TT}(\X) := H^\bullet_\TT(I\X)
\]
with a grading and product defined as follows.  Let
$R_{\TT}:=\Sym^\bullet_\CC(N^*\otimes \CC) =
H_{\TT}^{2\bullet}(\text{pt})$, noting that elements of $H^{2k}({\rm
  pt})$ are taken to have degree~$k$. As an $R_{\TT}$-module, we have:
\begin{equation}
  \label{eq:HCRX}
  H^\bullet_{\CR, \TT}(\X) \cong
  \frac{R_{\TT}[N]}{
    \left\{\chi - \sum_{i=1}^n \chi(\rho_i)y^{\rho_i} : 
      \chi \in N^* \otimes \CC \cong H^2_{\TT}({\rm pt}) \right \}}
\end{equation}
This is a graded ring with respect to the Chen--Ruan orbifold cup
product \cite{BCS,JT,ccliu}; here if $b \in N$ is such that
$\bar{b}=\sum_{\bar{\rho}_i \in \sigma} m_i \bar{\rho}_i$ where
$\sigma$ is the minimal cone in containing $\bar{b}$, then $y^b$ has
degree $\sum_{\bar{\rho}_i \in \sigma} m_i$.  The degree of $y^b$ is
known as the \emph{age} of $b$.  For $b\in \Box(\bSigma)$, the unit
class supported on the component $I\X(\bSigma)_b$ of the inertia stack
corresponds under \eqref{eq:HCRX} to $y^b$.  The fact that $I\X_0 =
\X$ gives a canonical inclusion $H^\bullet_\TT(\X;\CC) \subset
H^\bullet_{\CR, \TT}(\X)$, and the class $u_i \in H^2_\TT(\X)$ given
by the $\TT$-equivariant Poincar\'e-dual to the $i$th toric divisor in
$\X$ corresponds under \eqref{eq:HCRX} to $y^{\rho_i}$.

\subsection{Extended Stacky Fans}
\label{sec:extended_stacky_fans}
Let $\bSigma=(N,\Sigma,\rho)$ be a stacky fan, write $N_\Sigma := \{
c\in N : \bar{c} \in |\Sigma|\}$, and let $S$ be a finite set equipped
with a map $S \to N_\Sigma$.  We label the finite set $S$ by
$\{1,\dots,m\}$, where $m=|S|$, and write $s_j\in N$ for the image of
the $j$th element of $S$.  The {\em $S$-extended stacky fan} is
$(N,\Sigma,\rho^S)$ where $\rho^S\colon\ZZ^{n+m}\to N$ is defined by:
\begin{equation*}
\rho^S(e_i) = 
\begin{cases}
  \rho_i &  1\leq i\leq n \\
  s_{i-n} & n <  i \leq n+m
\end{cases}
\end{equation*}
and $e_i$ denotes the $i$th standard basis vector for $\ZZ^n$.  This
gives an $S$-extended fan sequence 
\[
\begin{CD}
0 @>>> \LL^S @>>> \ZZ^{n+m} 
@>{\rho^{S}}>> N
\end{CD}
\]
and, by Gale duality, an $S$-extended divisor sequence:
\[
\begin{CD}
0 @>>> N^* @>>> (\ZZ^*)^{n+m} 
@>{\rho^{S\vee}}>> \LL^{S\vee}
\end{CD}
\]
The toric Deligne--Mumford stacks associated to the stacky fan $(N,
\Sigma, \rho)$ and the $S$-extended stacky fan $(N, \Sigma, \rho^S)$
are canonically isomorphic \cite{Jiang}.

\subsection{Extended Degrees for Toric Stacks}

Consider an $S$-extended stacky fan $\bSigma$ as in
\S\ref{sec:extended_stacky_fans}, and let $\X$ be the corresponding
toric Deligne--Mumford stack.  The inclusion $\ZZ^n \to \ZZ^{n+m}$ of
the first~$n$ factors induces an exact sequence: 
\[
\begin{CD} 
0 @>>> \LL @>>> \LL^S @>>> \ZZ^m 
\end{CD}
\]
This splits over $\QQ$, via the map $\mu \colon \QQ^m \to \LL^S
\otimes \QQ$ that sends the $j$th standard basis vector to
\begin{equation} 
\label{eq:splitting_vector}
e_{j+n} - \sum_{i : \bar{\rho}_i \in \sigma(j)} s_{ji} e_i  \in \LL^S\otimes \QQ 
\subset \QQ^{n+m} 
\end{equation} 
where $\sigma(j)$ is the minimal cone containing $\bar{s}_j$ and the
positive numbers $s_{ji}$ are determined by $\sum_{i: \bar{\rho}_i \in \sigma(j)}
s_{ji} \bar{\rho}_i = \bar{s}_j$. Thus we obtain an isomorphism:
\begin{equation} 
\label{eq:LS_splitting} 
\LL^S \otimes \QQ   \cong (\LL\otimes \QQ)  \oplus \QQ^m 
\end{equation} 
Recall that $\Pic(\X) \cong \LL^\vee$, and hence that the Mori cone
$\NE(\X)$ is a subset of $\LL\otimes \RR$.  The \emph{$S$-extended
Mori cone} is the subset of $\LL^S \otimes \RR$ given by:
\begin{align*}
  \NE^S(\X) = \NE(\X) \times (\RR_{\ge 0})^m 
  && \text{via \eqref{eq:LS_splitting}.}
\end{align*}
The $S$-extended Mori cone can be thought of as the cone spanned by
the ``extended degrees'' of certain orbifold stable maps $f \colon \cC
\to \X$: see \cite[\S4]{ccit2}.

\begin{notation}
  We denote the fractional part of $x$ by  $\langle x \rangle$.
\end{notation}

\begin{defn}
  Recall that $\LL^S \subset \ZZ^{n+m}$, where $m= |S|$.  For a cone
  $\sigma \in \Sigma$, denote by $\Lambda_\sigma^S\subset \LL^S
  \otimes \QQ$ the subset consisting of elements
  \[
  \lambda=\sum_{i=1}^{n+m}\lambda_i e_i
  \]
  such that $\lambda_{n+j}\in \ZZ, 1\leq j\leq m$, and $\lambda_i\in
  \ZZ$ if $\bar{\rho}_i \notin \sigma$ and $i \leq n$. Set
  $\Lambda^S:=\bigcup_{\sigma\in \Sigma} \Lambda^S_{\sigma}$.
\end{defn}

\begin{defn} 
  The \emph{reduction function} is
  \begin{equation*}
    \begin{aligned}
      v^S \colon \Lambda^S & \longrightarrow  \Box(\bSigma) \\
      \lambda & \longmapsto \sum_{i=1}^n
\lceil\lambda_i\rceil \rho_i+\sum_{j=1}^m\lceil \lambda_{n+j}\rceil s_j
    \end{aligned}
  \end{equation*} 
  The reduction function takes values in $\Box(\bSigma)$: for $\lambda
  \in \Lambda^S_\sigma$ we have $\overline{v^S(\lambda)} =
  \sum_{i=1}^n \langle -\lambda_i \rangle \bar{\rho}_i \in
  \sigma$. 
\end{defn}

\begin{defn}
  For a box element $b\in \Box(\bSigma)$, we set:
  \[
  \Lambda^S_b := \{\lambda \in \Lambda^S: v^S(\lambda) = b\} 
  \]
  and define:
  \begin{align*}
    \Lambda E^S & := \Lambda^S \cap \NE^S(\X) &
    \Lambda E^S_b & := \Lambda^S_b \cap \NE^S(\X) 
  \end{align*}
\end{defn}

\begin{notation}
  \label{notation:Qtilde}
  Recall that $Q^d$ denotes the representative of $d \in H_2(X;\ZZ)$
  in the Novikov ring $\Novikov$.  Given $\lambda \in \Lambda E^S$
  write $\lambda = (d, k)$ via \eqref{eq:LS_splitting}, so that $d \in
  \NE(\X) \cap H_2(X,\ZZ)$ and $k\in (\ZZ_{\ge 0})^m$.  We set:
  \[
  \tQ^\lambda = Q^d x^k = Q^d x_1^{k_1} \cdots x_m^{k_m} 
  \in \Novikov[\![x_1,\ldots,x_m]\!] 
  \]
\end{notation}

\subsection{Mirror Theorem} Once again, consider an $S$-extended
stacky fan $\bSigma$ as in \S\ref{sec:extended_stacky_fans}.  Let $\X$
be the corresponding toric Deligne--Mumford stack.

\begin{defn}
  The {\em $S$-extended $\TT$-equivariant $I$-function} of $\X$ is:
  \begin{equation*}
    I^S(t,x,z):=
    ze^{\sum_{i=1}^n u_i t_i/z}
    \sum_{b\in \Box(\bSigma)}
    \sum_{\lambda\in \Lambda E^S_b} 
    \tQ^\lambda e^{\lambda t} 
    \left(
      \prod_{i=1}^{n+m}
      \frac{\prod_{\<a\>=\< \lambda_i \>, a \leq 0}(u_i+a z)}
      {\prod_{\<a \>=\< \lambda_i \>, a\leq \lambda_i} (u_i+a z)} 
    \right) 
    y^b 
  \end{equation*}
  Here:
  \begin{itemize}
  \item $t=(t_1,\dots,t_n)$ are variables, and $e^{\lambda
      t}:=\prod_{i=1}^n e^{(u_i\cdot d) t_i}$.
  \item $x=(x_1,\ldots,x_m)$ are variables: see Notation~\ref{notation:Qtilde}.
  \item for each $\lambda \in \Lambda E^S_b$, we write $\lambda_i$ for
    the $i$th component of $\lambda$ as an element of $\QQ^{n+m}$; in
    particular $\<\lambda_i\>=0$ for $n < i\le n+m$.
  \item For $1 \leq i \leq n$, $u_i$ is the $\TT$-equivariant
    Poincar\'e dual to the $i$th toric divisor: see
    Section~\ref{sec:CR}.  For $n < i\le n+m$, $u_i$ is defined to
    be zero.
  \item $y^b$ is the unit class supported on the component of inertia
    $I\X(\bSigma)_b$ associated to $b\in \Box(\bSigma)$: see
    Section~\ref{sec:CR}.
  \end{itemize}
  The $I$-function $I^S(t,x,z)$ is a formal power series in $Q$, $x$,
  $t$ with coefficients in $H^\bullet_{\CR, \TT}(\X)(\!(z^{-1})\!)$.
\end{defn}

\begin{thm}[The mirror theorem for toric Deligne--Mumford stacks
  \cite{ccit2,ccfk}]
  \label{thm:toric_mirror_theorem}
  Let $\bSigma = (N,\Sigma,\rho)$ be a stacky fan, and let $\X$ be the
  corresponding toric Deligne--Mumford stack.  Let $S$ be a finite set
  equipped with a map to $N_\Sigma$. Suppose that 
  the coarse moduli space of $\X$ is semi-projective (projective
  over affine).  Then $I^S(t,x,{-z}) \in \sL$.
\end{thm}

\begin{rmk}
  The statement that $I^S(t,x,{-z}) \in \sL$ has a precise meaning in
  formal geometry: see \S2.3 and Theorem~31 in \cite{ccit2}.  The
  reader may want to work with a slightly vague but more intuitive
  interpretation of this statement: that $I^S(t,x,{-z}) \in \sL$ for
  all values of the parameters $t$ and $x$.  No confusion should
  result, and the statements that we make are valid in the above,
  precise sense.
\end{rmk}

\begin{rmk}
  In the work of Cheong--Ciocan-Fontanine--Kim \cite{ccfk}, a toric
  orbifold $\X$ is represented by a triple $(V, T, \theta)$ where $T$
  is a torus, $V$ is a representation of $T$, and $\theta$ is a a
  character of $T$ such that $V$ has no strictly $\theta$-semistable
  points.  The toric orbifold $\X$ is the stack quotient $[V
  /\!\!/_\theta T]$.  In this language, $S$-extending the stacky fan
  $\bSigma$ corresponds to changing the GIT presentation $(V, T,
  \theta)$ of $\X$.  Thus Theorem~\ref{thm:toric_mirror_theorem} with
  non-trivial $S$-extension (not just
  Theorem~\ref{thm:toric_mirror_theorem} with $S=\varnothing$) can be
  obtained from \cite{ccfk} by considering an appropriate GIT
  presentation of $\X$.
\end{rmk}

\subsection{Condition $\sharp$ and Condition $S$-$\sharp$}
\label{sec:sharp}

Recall that the $J$-function \eqref{eq:J_function} is characterized by
the fact that $J_\X(t,-z)$ is the unique point on $\sL$ of the form
${-z} + t + O(z^{-1})$.  Let $\bSigma$ be a stacky fan and let $S$ be
a finite set equipped with a map $\kappa \colon S \to \Box(\bSigma)$.
Label the elements of $S$ by $\{1,\ldots,m\}$, where $m = |S|$, and
let $s_j = \kappa(j)$.  We say that the $S$-extended stacky fan
$\bSigma$ satisfies \emph{condition $S$-$\sharp$} if and only if
\[
I^S(t,x,{-z}) = {-z} + t + \sum_{j = 1}^m x_j y^{s_j} + O(z^{-1})
\]
If condition $S$-$\sharp$ holds then $J_\X(\tau,z) = I^S(t,x,z)$ where
$\tau = t + \sum_{j = 1}^m x_j y^{s_j}$.  We say that a stacky fan
$\bSigma$ satisfies \emph{condition $\sharp$} if and only if it
satisfies condition $S$-$\sharp$ with $S = \varnothing$.

Condition $\sharp$ is equivalent to the statement: for all $b \in
\Box(\bSigma)$ and all non-zero $\lambda \in \Lambda E^\varnothing_b$
we have:
\[
-K_\X \cdot \lambda +\age(b) +\# \bigl\{i\mid \text{$\lambda_i<0$
 and $\lambda_i\in \ZZ$}\bigr\} \geq 2
\]
This is not automatically satisfied for Fano stacks or even for Fano
orbifolds: see \S\ref{sec:toric-surface}.  The surface ${\mathbb F}_2$
is nonsingular and weak Fano and it does not satisfy
condition~$\sharp$.

\begin{lem}[A simple criterion for condition $\sharp$ to hold]
  \label{lem:sharp_fano}
  Let $\bSigma = (N,\Sigma,\rho)$ be a stacky fan, and let $\X$ be the
  corresponding toric Deligne--Mumford stack.  Suppose that $\X$ is
  smooth and has semi-projective coarse moduli space.  If $\X$ is Fano
  and has \emph{canonical singularities}, that is, if $\age(b)\geq 1$
  for all non-zero elements $b \in \Box(\bSigma)$, then $\X$ satisfies
  condition~$\sharp$.
\end{lem}

\begin{proof}
  The key observation is that if $\X$ is a toric stack and $\lambda \in
  \Lambda E^\varnothing_b\setminus\{0\}$ 
is the class of a \emph{compact} curve, then 
  there are at least two toric divisors $D_i\subset \X$ such that
  $\lambda\cdot D_i>0$.  Set:
  \[
  N_\lambda=\# \bigl\{i\mid \text{$\lambda_i<0$ and $\lambda_i\in
    \ZZ$}\bigr\}
  \]
  The quantity $\ord \lambda :=-K_\X \cdot \lambda +\age(b)
  +N_\lambda$ is an integer.  If $N_l\geq 1$ then, since $\X$ is Fano,
  $-K_\X\cdot \lambda>0$ and $\ord \lambda\geq 2$.  If $N_\lambda=0$
  and $\age(b)\geq 1$, the same argument applies. Otherwise $b=0$, and
  then $\lambda_i\in \ZZ$ for all $i$, and all $\lambda_i\geq 0$. By
  the key observation, at least two of the $\lambda_i$ are strictly
  positive and we are done.
\end{proof}

\begin{rmk}
  Note that the criterion in Lemma~\ref{lem:sharp_fano} is far from
  best possible: the more positive the anticanonical class is, the
  worse the singularities are allowed to be.
\end{rmk}
\section{Applying the Mirror Theorem}
\label{sec:examples}

\subsection{Example 1: $B\mu_3$}

This is the toric Deligne--Mumford stack $\X$ associated to the stacky
fan $\bSigma = (N,\Sigma,\rho)$, where $N = \frac{1}{3}\ZZ/\ZZ$,
$\Sigma = \{0\}$, and $\rho \colon (0)\to N$ is the zero map.  We have
$\Box(\bSigma)=\bigl\{0,\frac{1}{3},\frac{2}{3} \bigr\}$.  We consider
the $S$-extended $I$-function where $S = \Box(\bSigma)$ and $S \to
N_\Sigma$ is the canonical inclusion.  The $S$-extended fan map is:
\[
\rho^S = 
\begin{pmatrix}
  0 & \frac{1}{3} & \frac{2}{3} 
\end{pmatrix}
\colon \ZZ^3 \to N
\]
so that $\LL^S_\QQ=\QQ^3$ and $\LL^S$ is the lattice of vectors:
\[
\begin{pmatrix}
  k_0\\
  k_1\\
  k_2
\end{pmatrix}
\in \ZZ^3
\quad
\text{such that}
\quad
k_1+2k_2 \equiv 0 \bmod 3
\]
The $S$-extended Mori cone is the positive octant.  We have $\Lambda^S =
\ZZ^3$, and the reduction function is
\[
v^S\colon 
\begin{pmatrix}
  k_0\\
  k_1\\
  k_2
\end{pmatrix}
\mapsto
\Bigl\langle \frac{k_1}{3}+\frac{2k_2}{3}\Bigr\rangle.
\]
The $S$-extended $I$-function is:
\[
I^S(x,z) = 
z \sum_{k_1=0}^\infty
\sum_{k_2=0}^\infty
\sum_{k_3=0}^\infty
\frac{x_0^{k_0}x_1^{k_1}x_2^{k_2}}{z^{k_0+k_1+k_2}k_0!k_1!k_2!} 
\fun_{\langle \frac{k_1}{3}+\frac{2k_2}{3}\rangle} 
\]
This is homogeneous of degree~$1$ if we set $\deg x_0=\deg x_1=\deg
x_2=\deg z =1$.  Since:
\[
I^S(x,z) = z + x_0 \fun_0 + x_1 \fun_{\frac{1}{3}} + x_2 \fun_{\frac{2}{3}}
+ O(z^{-1})
\]
condition $S$-$\sharp$ holds, and
Theorem~\ref{thm:toric_mirror_theorem} implies that:
\[
J_\X\big(x_0 \fun_0 + x_1 \fun_{\frac{1}{3}} + x_2 \fun_{\frac{2}{3}},z\big) = I^S(x,z)
\]

\subsection{Example 2: $\frac{1}{3}(1,1)$}

This is the toric Deligne--Mumford stack $\X$ associated to the stacky
fan $\bSigma = (N,\Sigma,\rho)$, where:
\[
\rho = 
\begin{pmatrix}
  1 & 0 \\
  0 & 1
\end{pmatrix}
\colon \ZZ^2 \to N=\ZZ^2+\textstyle \frac{1}{3}(1,1)\ZZ.
\]
and $\Sigma$ is the positive quadrant in $N_\QQ$.  We have
$\Box(\bSigma)=\bigl\{0,\frac{1}{3}(1,1),\frac{2}{3}(1,1) \bigr\}$; to
streamline the notation we will identify $\Box(\bSigma)$ with the set
$\bigl\{0,\frac{1}{3},\frac{2}{3}\bigr\}$ via the map $\kappa$ that
sends $x$ to $x(1,1)$.  We consider the $S$-extended $I$-function
where $S = \bigl\{0,\frac{1}{3}\bigr\}$ and $S$ maps to $N_\Sigma$ via
$\kappa$.  The $S$-extended fan map is:
\[
\rho^S =
\begin{pmatrix}
  1 & 0 & 0 & \frac{1}{3} \\
  0 & 1 & 0 & \frac{1}{3} 
\end{pmatrix}
\colon \ZZ^{2+2} \to N
\]
so that $\LL^S_\QQ\cong \QQ^2$ is identified as a subset of
$\QQ^{2+2}$ via the inclusion:
\[
\begin{pmatrix}
  k_0\\
  k_1
\end{pmatrix}
\mapsto
\begin{pmatrix}
  0 & -\frac{1}{3}  \\
  0 & -\frac{1}{3}  \\
  1 & 0             \\
  0 & 1               
\end{pmatrix}
\begin{pmatrix}
  k_0\\
  k_1
\end{pmatrix}
\] 
The $S$-extended Mori cone is the positive quadrant.  We see that
$\Lambda^S \subset \LL^S_\QQ$ is the lattice of vectors:
\[
\begin{pmatrix}
  k_0\\
  k_1
\end{pmatrix}
\quad \text{such that $k_0$,~$k_1 \in \ZZ$}
\]
and that the reduction function is:
\[
v^S\colon 
\begin{pmatrix}
  k_0\\
  k_1
\end{pmatrix}
\mapsto
\Bigl\langle \frac{k_1}{3}\Bigr\rangle
\]
The $S$-extended $I$-function is:
\[
I^S(t,x,z) = 
z e^{(u_1 t_1 + u_2 t_2)/z}
\sum_{k_0=0}^\infty \sum_{k_1 = 0}^\infty
\frac{x_0^{k_0}x_1^{k_1}}{z^{k_0+k_1}k_0!k_1!} 
\mathbf{1}_{\langle \frac{k_1}{3}\rangle}
\prod_{\substack{\langle b\rangle = \langle
    -\frac{k_1}{3}\rangle\\-\frac{k_1}{3}<b\leq
    0}}(u_1+bz)(u_2+bz) . 
\]
This is homogeneous of degree~$1$ if we set $\deg t_1=\deg t_2=0$,
$\deg x_0=\deg z =1$, and $\deg x_1 = \frac{1}{3}$.
Theorem~\ref{thm:toric_mirror_theorem} gives that $I^S(x,-z) \in
\sL_\X$, and we have:
\[
I^S(t,x,z) = z + t_1 u_1 + t_2 u_2 + x_0 \fun_0 + x_1 \fun_{\frac{1}{3}} + O(z^{-1})
\]
Thus condition $S$-$\sharp$ holds, and we obtain an expression for the
$J$-function of $\X$:
\[
J_\X\big( t_1 u_1 + t_2 u_2 + x_0 \fun_0 + x_1 \fun_{\frac{1}{3}},z\big) = I^S(t,x,z)
\]

\subsection{Example 3: $\PP(1,1,3)$}

This is the toric Deligne--Mumford stack $\X$ associated to the stacky
fan $\bSigma = (N,\Sigma,\rho)$, where:
\[
\rho = 
\begin{pmatrix}
  1 & 0 & -\frac{1}{3}\\
  0 & 1 & -\frac{1}{3}
\end{pmatrix}
\colon \ZZ^3 \to N=\ZZ^2+\frac{1}{3}(1,1)\ZZ 
\]
and $\Sigma$ is the complete fan in $N_\QQ \cong \QQ^2$ with rays
given by the columns of $\rho$. We identify
$\Box(\bSigma)=\bigl\{0,\frac{1}{3}(1,1),\frac{2}{3}(1,1) \bigr\}$
with the set $\bigl\{0,\frac{1}{3},\frac{2}{3}\bigr\}$ via the map
$\kappa$ that sends $x$ to $x(1,1)$.  We consider the $S$-extended
$I$-function where $S = \bigl\{0,\frac{1}{3}\bigr\}$ and $S$ maps to
$N_\Sigma$ via $\kappa$.  The $S$-extended fan map is:
\[
\rho^S =
\begin{pmatrix}
  1 & 0 & -\frac{1}{3} & 0 & \frac{1}{3} \\
  0 & 1 & -\frac{1}{3} & 0 & \frac{1}{3} 
\end{pmatrix}
\colon \ZZ^{3+2} \to N
\]
so that $\LL^S_\QQ\cong \QQ^3$ is identified as a subset of
$\QQ^{3+2}$ via the inclusion:
\[
\begin{pmatrix}
  l  \\
  k_0\\
  k_1
\end{pmatrix}
\mapsto
\begin{pmatrix}
  \frac{1}{3} & 0 & -\frac{1}{3} \\
  \frac{1}{3} & 0 & -\frac{1}{3} \\
  1 & 0 & 0 \\
  0 & 1 & 0 \\
  0 & 0 & 1 
\end{pmatrix}
\begin{pmatrix}
  l  \\
  k_0\\
  k_1
\end{pmatrix}
\] 
The $S$-extended Mori cone is the positive octant. We see that
$\Lambda^S\subset \LL^S_\QQ$ is the lattice of vectors:
\[
\begin{pmatrix}
  l \\
  k_0\\
  k_1
\end{pmatrix}
\quad
\text{such that $l$,~$k_0$,~$k_1 \in \ZZ$}
\]
and that the reduction function is:
\[
v^S\colon 
\begin{pmatrix}
  l \\
  k_0\\
  k_1
\end{pmatrix}
\mapsto
\Bigl\langle-\frac{l}{3} +\frac{k_1}{3} \Bigr\rangle
\]
Let us identify the Novikov ring $\Novikov$ with $\CC[\![Q]\!]$ via the
map that sends $d \in H_2(X;\ZZ)$ to $Q^{\int_d c_1(\cO(3))}$.  The
$S$-extended $I$-function is:
\begin{multline*}
  I^S(t,x,z) = 
  z e^{(u_1 t_1 + u_2 t_2 + u_3 t_3)/z} \\
  \times
  \sum_{l=0}^\infty 
  \sum_{k_1=0}^\infty
  \sum_{k_2=0}^\infty
  \frac{Q^l x_0^{k_0} x_1^{k_1} e^{(t_1+t_2+3t_3)l}}{z^{k_0+k_1} k_0! k_1!}
  \frac{\prod_{\substack{\langle b \rangle=\langle
        \frac{l}{3}-\frac{k_1}{3}\rangle\\ b\leq 0}}(u_1+bz)(u_2+bz)}{
    \prod_{\substack{\langle b\rangle = \langle
        \frac{l}{3}-\frac{k_1}{3}\rangle\\
        b\leq \frac{l}{3}-\frac{k_1}{3}}} (u_1+bz)(u_2+bz)}
  \frac{  \fun_{\langle -\frac{l}{3}+\frac{k_1}{3}\rangle}
  }{\prod_{0<b\leq l}(u_3+bz)}
\end{multline*}
This is homogeneous of degree $1$ if we set $\deg t_1 = \deg t_2 =
\deg t_3 = 0$, $\deg x_0 = \deg z = 1$, $\deg x_1 = \frac{1}{3}$, and
$\deg Q = \frac{5}{3}$.  Theorem~\ref{thm:toric_mirror_theorem}
gives that $I^S(x,t,-z) \in \sL$, and we have:
\[
I^S(t,x,z) = z + t_1 u_1 + t_2 u_2 + t_3 u_3 + x_0 \fun_0 + x_1
\fun_{\frac{1}{3}} + O(z^{-1})
\]
Thus condition $S$-$\sharp$ holds, and we obtain an expression for the
$J$-function of $\X$:
\[
J_\X\big(z + t_1 u_1 + t_2 u_2 + t_3 u_3 + x_0 \fun_0 + x_1
\fun_{\frac{1}{3}},z\big) = I^S(t,x,z)
\]

\begin{rmk}
  Condition $\sharp$ holds for any weighted projective space, but
  condition $S$-$\sharp$ fails in general.  Indeed we chose $S =
  \bigl\{0,\frac{1}{3}\bigr\}$ here rather than $S = \Box(\bSigma) =
  \bigl\{0,\frac{1}{3}, \frac{2}{3}\bigr\}$ because with the latter
  choice condition $S$-$\sharp$ fails and we do not obtain a closed
  form expression for $J_\X$.  This is what we meant in
  \cite[Remark~34]{ccit2}.  
\end{rmk}

\begin{rmk}
  \label{rmk:P113_compare}
  The non-equivariant limit of our $S$-extended $I$-function, in the
  notation of \cite{cclt}, is:
  \begin{equation}
    \label{eq:P113_non_equivariant}
    z e^{(t_1 + t_2 + 3 t_3)P/z} \sum_{l=0}^\infty 
    \sum_{k_1=0}^\infty
    \sum_{k_2=0}^\infty
    \frac{Q^l x_0^{k_0} x_1^{k_1} e^{(t_1+t_2+3t_3)l}}{z^{k_0+k_1} k_0! k_1!}
    \frac{\prod_{\substack{\langle b \rangle=\langle
          \frac{l}{3}-\frac{k_1}{3}\rangle\\ b\leq 0}}(P+bz)^2}
    {\prod_{\substack{\langle b\rangle = \langle
          \frac{l}{3}-\frac{k_1}{3}\rangle\\
          b\leq \frac{l}{3}-\frac{k_1}{3}}} (P+bz)^2}
    \frac{  \fun_{\langle -\frac{l}{3}+\frac{k_1}{3}\rangle}
    }{\prod_{0<b\leq l}(3P+bz)}
  \end{equation}
  Theorem~\ref{thm:toric_mirror_theorem} implies that this lies on the
  Lagrangian submanifold $\sL^\non$ for non-equivariant
  Gromov--Witten theory of $\X$ and, since
  \eqref{eq:P113_non_equivariant} takes the form
  \[
  z + (t_1 + t_2 + 3 t_3) P  + x_0 \fun_0 + x_1
  \fun_{\frac{1}{3}} + O(z^{-1})
  \]
  we see that this determines the non-equivariant $J$-function
  $J_\X(t,x,z)$ for $t = t_1 P + x_0 \fun_0 + x_1 \fun_{\frac{1}{3}}$.
  Theorem~\ref{thm:toric_mirror_theorem} thus determines the orbifold
  quantum product $\star_t$, for $t$ as above, in a straightforward
  way.  This improves on the results of \cite{cclt}, which determine
  $J_\X(t,x,z)$ for $t$ in the small quantum cohomology locus $H^2(\X)
  \subset H^\bullet_\CR(\X)$ and thus determine the small quantum
  orbifold cohomology ring of $\X$.
\end{rmk}  

\begin{rmk}
  To determine the full big quantum orbifold cohomology ring of $\X$
  (equivariant or non-equivariant) from
  Theorem~\ref{thm:toric_mirror_theorem} is more involved.  One needs
  to take $S = \bigl\{0,\frac{1}{3}, \frac{2}{3}\bigr\}$, so in
  particular condition $S$-$\sharp$ fails, and then compute the big
  $J$-function $J_\X(t,z)$ by Birkhoff factorization, as in
  \S\ref{sec:P2_Birkhoff} below.  We do not know a closed-form
  expression for the structure constants.
\end{rmk}

\begin{rmk}
  Note that:
  \[
  I^S_{\PP(1,1,3)}(0,x,z)|_{Q=0}=I^S_{\frac{1}{3}(1,1)}(0,x,z)
  \]
  and that, as discussed in Remark~\ref{rmk:P113_compare},
  $I^S(0,0,z)$ essentially coincides, after passing to the non-equivariant
  limit and changing notation for degrees (replacing $d$ by
  $\frac{d}{3}$), with the \emph{small $I$-function} of $\PP(1,1,3)$ as
  written in \cite{cclt}.
\end{rmk}

\subsection{Example 4: $\PP(2,2)$}
\label{sec:P(2,2)}

This is the toric Deligne--Mumford stack $\X$ associated to the stacky
fan $\bSigma = (N,\Sigma,\rho)$, where:
\[
\rho = 
\begin{pmatrix}
  -1 & 1 \\
   0 & 1
\end{pmatrix}
\colon \ZZ^2 \to N=\ZZ \oplus (\ZZ/2\ZZ).
\]
and $\Sigma$ is the fan in $N_\QQ \cong \QQ$ with rays given by ${-1}$
and $1$. We identify $\Box(\bSigma)=\bigl\{(0,0),(0,1) \bigr\}$ with
the set $\bigl\{0,\frac{1}{2}\bigr\}$ via the map $\kappa$ that sends
$0$ to $(0,0)$ and $\frac{1}{2}$ to $(0,1)$.  We consider the
$S$-extended $I$-function where $S =
\bigl\{(0,0),(0,1),(-1,1),(1,0)\bigr\}$ and $S \to N_\Sigma$ is the
canonical inclusion.  The $S$-extended fan map is:
\[
\rho^S =
\begin{pmatrix}
  -1 & 1 & 0 & 0 & -1 & 1\\
   0 & 1 & 0 & 1 & 1  & 0
\end{pmatrix}
\colon \ZZ^{2+4} \to N
\]
so that $\LL^S_\QQ\cong \QQ^5$ is identified as a subset of
$\QQ^{2+4}$ via the inclusion:
\[
\begin{pmatrix}
  l  \\
  k_0\\
  k_1\\
  k_2\\
  k_3
\end{pmatrix}
\mapsto
\begin{pmatrix}
  1 & 0 & 0 & -1& 0\\
  1 & 0 & 0 & 0 &-1\\
  0 & 1 & 0 & 0 & 0\\
  0 & 0 & 1 & 0 & 0\\
  0 & 0 & 0 & 1 & 0\\
  0 & 0 & 0 & 0 & 1 
\end{pmatrix}
\begin{pmatrix}
  l  \\
  k_0\\
  k_1\\
  k_2\\
  k_3
\end{pmatrix}
\] 
The $S$-extended Mori cone is the positive orthant.
We see that $\Lambda^S \subset \LL^S_\QQ$ is the lattice of vectors:
\[
\begin{pmatrix}
  l\\
  k_0\\
  k_1\\
  k_2\\
  k_3
\end{pmatrix}
\quad
\text{such that 
$l$,~$k_0$,~$k_1$,~$k_2$,~$k_3 \in \ZZ$}
\]
and that the reduction function is:
\[
v^S\colon 
\begin{pmatrix}
  l\\
  k_0\\
  k_1\\
  k_2\\
  k_3
\end{pmatrix}
\mapsto
\Bigl\langle \frac{l+k_1+k_2+k_3}{2}\Bigr\rangle
\]
Let us identify the Novikov ring $\Novikov$ with $\CC[\![Q]\!]$ via the
map that sends $d \in H_2(X;\ZZ)$ to $Q^{\int_d c_1(\cO(2))}$.  The
$S$-extended $I$-function is:
\begin{multline*}
  I^S(t,x,z) = 
  z e^{(u_1 t_1 + u_2 t_2)/z}  \\
  \times
  \sum_{(l,k_0,\ldots,k_3) \in \NN^5}
  \frac{Q^l x_0^{k_0} x_1^{k_1} x_2^{k_2} x_3^{k_3} e^{(t_1+t_2)l}}
  {z^{k_0+k_1+k_2+k_3} k_0! k_1! k_2! k_3!}
  \frac{\prod_{b\leq 0}(u_1+bz)}{
    \prod_{b\leq l-k_2}(u_1+bz)}
  \frac{\prod_{b\leq 0}(u_2+bz)}{
    \prod_{b\leq l-k_3}(u_2+bz)}
  \fun_{\bigl\langle \frac{l+k_1+k_2+k_3}{ 2}\bigr \rangle}
\end{multline*}
This is homogeneous of degree $1$ if we set $\deg t_1 = \deg t_2 =
\deg x_2 = \deg x_3 = 0$, $\deg x_0 = \deg x_1 = \deg z = 1$, and
$\deg Q = 2$.  Theorem~\ref{thm:toric_mirror_theorem} gives that
$I^S(x,t,-z) \in \sL$, and straightforward calculation gives:
\[
I^S(t,x,z) = \textstyle z \fun_0 + \tau(x,t) + O(z^{-1})
\]
where:
\begin{multline}
  \label{eq:P22_tau}
  \tau(t,x) =  \textstyle x_0 \fun_0 + \big(t_1 + \frac{1}{2}
  \log(1-x_2^2) \big) u_1 \fun_0 + \big( t_2 + \frac{1}{2}
  \log(1-x_3^2) \big)  u_2 \fun_0 \\ 
  \textstyle
  + x_1 \fun_{\frac{1}{2}} + 
  \frac{1}{2} \log\left(
    \frac{1+x_2}{1-x_2} \right) u_1 \fun_{\frac{1}{2}} 
  + \frac{1}{2} \log\left(
    \frac{1+x_3}{1-x_3} \right) u_2 \fun_{\frac{1}{2}} 
\end{multline}
Thus $J_\X\big(\tau(t,x),z\big) = I^S(t,x,z)$.  We can invert the
mirror map $(x,t) \mapsto \tau(x,t)$ in closed form: if $\tau(x,t) =
a_0 \fun_0 + a_1 u_1 \fun_0 + a_2 u_2 \fun_0 + b_0 \fun_{\frac{1}{2}}
+ b_1 u_1 \fun_{\frac{1}{2}} + b_2 u_2 \fun_{\frac{1}{2}}$ then:
\begin{equation}
  \label{eq:P22_inverse_mirror_map}
  \begin{aligned}
    x_0 &= a_0 & x_1 &= b_0 & x_2 &= \tanh b_1\\
    x_3 &= \tanh b_2 & t_1 &= a_1 - \log \sech b_1 & t_2 &= a_2 - \log
    \sech b_2
  \end{aligned}
\end{equation}
This gives a closed-form expression for the $J$-function
$J_\X(\tau,z)$.

\begin{rmk}
  It is instructive to consider the specialisations of $I^S(t,x,z)$ to
  $Q=x_2=x_3=0$ and to $x_0=x_1=x_2=x_3=0$. Note that $\PP(2,2)$
  satisfies condition $\sharp$ but not condition $S$-$\sharp$.
\end{rmk}

\subsection{Example 5: $\PP^1 \times B\mu_2$}
\label{sec:PP1_times_Bmu2}
This is the toric Deligne--Mumford stack $\X$ associated to the stacky
fan $\bSigma = (N,\Sigma,\rho)$, where:
\[
\rho = 
\begin{pmatrix}
  -1 & 1 \\
  0 & 0
\end{pmatrix}
\colon \ZZ^2 \to N=\ZZ\oplus(\ZZ/2\ZZ).
\]
and $\Sigma$ is the fan in $N_\QQ \cong \QQ$ with rays given by ${-1}$
and $1$. We identify $\Box(\bSigma)=\bigl\{(0,0),(0,1) \bigr\}$ with
the set $\bigl\{0,\frac{1}{2}\bigr\}$ via the map $\kappa$ that sends
$0$ to $(0,0)$ and $\frac{1}{2}$ to $(0,1)$.  We consider the
$S$-extended $I$-function where $S =
\bigl\{(0,0),(0,1),(-1,1),(1,1)\bigr\}$ and $S \to N_\Sigma$ is the
canonical inclusion.  The $S$-extended fan map is:
\[
\rho^S =
\begin{pmatrix}
  -1 & 1 & 0 & 0 & -1 & 1  \\
   0 & 0 & 0 & 1 & 1 & 1
\end{pmatrix}
\colon \ZZ^{2+4} \to N
\]
so that $\LL^S_\QQ\cong \QQ^5$ is identified as a subset of
$\QQ^{2+4}$ via the inclusion:
\[
\begin{pmatrix}
  l  \\
  k_0\\
  k_1\\
  k_2\\
  k_3
\end{pmatrix}
\mapsto
\begin{pmatrix}
  1 & 0 & 0 & -1 & 0\\
  1 & 0 & 0 &  0 &-1\\
  0 & 1 & 0 &  0 & 0\\
  0 & 0 & 1 &  0 & 0\\
  0 & 0 & 0 &  1 & 0\\
  0 & 0 & 0 &  0 & 1
\end{pmatrix}
\begin{pmatrix}
  l  \\
  k_0\\
  k_1\\
  k_2\\
  k_3
\end{pmatrix}
\] 
The $S$-extended Mori cone is the positive orthant.  We see that
$\Lambda^S \subset \LL^S_\QQ$ is the lattice of vectors:
\[
\begin{pmatrix}
  l\\
  k_0\\
  k_1\\
  k_2\\
  k_3
\end{pmatrix}
\quad \text{such that $l$,~$k_0$,~$k_1$,~$k_2$,~$k_3\in \ZZ$}
\]
and that the reduction function is:
\[
v^S\colon
\begin{pmatrix}
  l\\
  k_0\\
  k_1\\
  k_2\\
  k_3
\end{pmatrix}
\mapsto
\Bigl\langle \frac{k_1+k_2+k_3}{2}\Bigr\rangle.
\]
Let us identify the Novikov ring $\Novikov$ with $\CC[\![Q]\!]$ via the
map that sends $d \in H_2(X;\ZZ)$ to $Q^{\int_d c_1(\cO_{\PP^1}(1))}$.  The
$S$-extended $I$-function is:
\begin{multline*}
  I^S(t,x,z) = 
  z e^{(u_1 t_1 + u_2 t_2)/z}  \\
  \times
  \sum_{(l,k_0,\ldots,k_3) \in \NN^5}
  \frac{Q^l x_0^{k_0} x_1^{k_1} x_2^{k_2} x_3^{k_3} e^{(t_1+t_2)l}}
  {z^{k_0+k_1+k_2+k_3} k_0! k_1! k_2! k_3!}
  \frac{\prod_{b\leq 0}(u_1+bz)}{
    \prod_{b\leq l-k_2}(u_1+bz)}
  \frac{\prod_{b\leq 0}(u_2+bz)}{
    \prod_{b\leq l-k_3}(u_2+bz)}
  \fun_{\bigl\langle \frac{k_1+k_2+k_3}{ 2}\bigr \rangle}
\end{multline*}
Except for the difference in reduction function, this coincides with
the $S$-extended $I$-function for $\PP(2,2)$ in \S\ref{sec:P(2,2)}.  Once
again, Theorem~\ref{thm:toric_mirror_theorem} gives
that $I^S(x,t,-z) \in \sL$, and:
\[
I^S(t,x,z) = \textstyle z \fun_0 + \tau(x,t) + O(z^{-1})
\]
with $\tau(x,t)$ as in \eqref{eq:P22_tau}.  Thus
$J_\X\big(\tau(t,x),z\big) = I^S(t,x,z)$.  Inverting the mirror map
\eqref{eq:P22_inverse_mirror_map} gives a closed-form expression for
the $J$-function $J_\X(\tau,z)$.

\subsection{Example 6: $\PP_{2,2}$}
\label{sec:P22}
This is the unique Deligne--Mumford stack with coarse moduli space
equal to $\PP^1$, isotropy group $\mu_{2}$ at $0 \in \PP^1$, isotropy
group $\mu_{2}$ at $\infty \in \PP^1$, and no other non-trivial
isotropy groups.  It is the toric Deligne--Mumford stack $\X$
associated to the stacky fan $\bSigma = (N,\Sigma,\rho)$, where:
\[
\rho = 
\begin{pmatrix}
  -1 & 1 
\end{pmatrix}
\colon \ZZ^2 \to N=\ZZ+\frac{1}{2}\ZZ.
\]
and $\Sigma$ is the fan in $N_\QQ \cong \QQ$ with rays given by ${-1}$
and $1$. We identify $\Box(\bSigma)$ with the set
$\bigl\{(0,0),\bigl(\frac{1}{2},0\bigr), \bigl(0,\frac{1}{2}\bigr)
\bigr\}$ via the map $\rho$.  We consider the $S$-extended
$I$-function where $S=\Box(\bSigma)$ and $S \to N_\Sigma$ is the
canonical inclusion.  The $S$-extended fan map is:
\[
\rho^S =
\begin{pmatrix}
  -1 & 1 & 0 & -\frac{1}{2} & \frac{1}{2} 
\end{pmatrix}
\colon \ZZ^{2+3} \to N
\]
so that $\LL^S_\QQ\cong \QQ^4$ is identified as a subset of
$\QQ^{2+3}$ via the inclusion:
\[
\begin{pmatrix}
  l  \\
  k_0\\
  k_1\\
  k_2
\end{pmatrix}
\mapsto
\begin{pmatrix}
  \frac{1}{2} & 0 & -\frac{1}{2}& 0           \\
  \frac{1}{2} & 0 & 0           & -\frac{1}{2}\\
  0 & 1 & 0 & 0\\
  0 & 0 & 1 & 0\\
  0 & 0 & 0 & 1
\end{pmatrix}
\begin{pmatrix}
  l  \\
  k_0\\
  k_1\\
  k_2
\end{pmatrix}
\] 
The $S$-extended Mori cone is the positive orthant.  We see that
$\Lambda^S \subset \LL^S_\QQ$ is the subset (\emph{not} sublattice) of
vectors:
\[
\begin{pmatrix}
  l\\
  k_0\\
  k_1\\
  k_2
\end{pmatrix}
\quad \text{such that $l$,~$k_0$,~$k_1$,~$k_2 \in \ZZ$ and at least
  one of $l-k_1$,~$l-k_2$ is even}
\]
and that the reduction function is:
\[
v^S\colon 
\begin{pmatrix}
  l\\
  k_0\\
  k_1\\
  k_2
\end{pmatrix}
\mapsto
\Bigl( \textstyle \bigl\langle \frac{k_1-l}{2}\bigr\rangle,
\bigl\langle \frac{k_2-l}{2}\bigr\rangle\Bigr).
\]
Let us identify the Novikov ring $\Novikov$ with $\CC[\![Q]\!]$ via the
map that sends $d \in H_2(X;\ZZ)$ to $Q^{\int_d c_1(\cO_X(1))}$.  The
$S$-extended $I$-function is:
\begin{multline*}
  I^S(t,x,z) = 
  z e^{(u_1 t_1 + u_2 t_2)/z}  \\
  \times
  \sum_{(l,k_0,k_1,k_2) \in \Lambda^S}
  \frac{Q^l x_0^{k_0} x_1^{k_1} x_2^{k_2} e^{(t_1+t_2)l}}
  {z^{k_0+k_1+k_2} k_0! k_1! k_2!}
  \frac{\prod_{\substack{\langle b \rangle = \bigl\langle
        \frac{l-k_1}{ 2} \bigr
        \rangle \\ b\leq 0}}(u_1+bz)}{
    \prod_{\substack{\langle b \rangle = \bigl\langle
        \frac{l-k_1}{ 2} \bigr
        \rangle \\ b\leq \frac{l-k_1}{2}}}(u_1+bz)}
  \frac{\prod_{\substack{\langle b \rangle = \bigl\langle
        \frac{l-k_2}{ 2} \bigr
        \rangle \\ b\leq 0}}(u_2+bz)}{
    \prod_{\substack{\langle b \rangle = \bigl\langle
        \frac{l-k_2}{ 2} \bigr
        \rangle \\ b\leq \frac{l-k_2}{2}}}(u_2+bz)}
  \fun_{\bigl(\bigl\langle \frac{k_1-l}{ 2}\bigr \rangle, \bigl\langle \frac{k_2-l}{ 2}\bigr \rangle\bigr)}
\end{multline*}
This is homogeneous of degree $1$ if we set $\deg t_1 = \deg t_2 = 0$,
$\deg x_0 = \deg Q = \deg z = 1$, and $\deg x_1 = \deg x_2 =
\frac{1}{2}$.  Theorem~\ref{thm:toric_mirror_theorem} gives that
$I^S(x,t,-z) \in \sL$, and since:
\[
I^S(x,t,z) = z \fun_{(0,0)} + t_1 u_1 \fun_{(0,0)} + t_2 u_2
\fun_{(0,0)} + x_0 \fun_{(0,0)} + x_1 \fun_{(\frac{1}{2},0)} + x_2
\fun_{(0,\frac{1}{2})} + O(z^{-1})
\]
we conclude that:
\[
J_\X\big(t_1 u_1 \fun_{(0,0)} + t_2 u_2 \fun_{(0,0)} + x_0 \fun_{(0,0)} +
x_1 \fun_{(\frac{1}{2},0)} + x_2 \fun_{(0,\frac{1}{2})},z\big) = 
I^S(x,t,z)
\]

\subsection{Example 7: a toric surface}
\label{sec:toric-surface}

We have already seen examples (in \S\ref{sec:P(2,2)} and
\S\ref{sec:PP1_times_Bmu2}) where condition $S$-$\sharp$ fails.  We
now give the simplest example of a \emph{Fano} toric stack such that
condition~$\sharp$ fails.  Consider the toric Deligne--Mumford stack
$\X$ associated to the stacky fan $\bSigma = (N,\Sigma,\rho)$, where:
\[
\rho = 
\begin{pmatrix}
  1 & 0 & -1 &  0 \\
  0 & 1 & -3 & -2
\end{pmatrix}
\colon \ZZ^2 \to N=\ZZ^2
\]
and $\Sigma$ is the complete fan in $N_\QQ \cong \QQ^2$ with rays
given by the columns of $\rho$.  We identify
$\Box(\bSigma)=\bigl\{(0,0),(0,-1)\bigr\}$ with the set
$\bigl\{0,\frac{1}{2}\bigr\}$ via the map $\kappa$ that sends $x$ to
$(0,-2x)$.  We identify $\LL_\QQ\cong \QQ^2$ as a subset of $\QQ^{4}$
via the inclusion:
\[
\begin{pmatrix}
  l_1\\
  l_2
\end{pmatrix}
\mapsto
\begin{pmatrix}
  1 & 0 \\
  3 & 2 \\
  1 & 0 \\
  0 & 1 
\end{pmatrix}
\begin{pmatrix}
  l_1\\
  l_2
\end{pmatrix}
\] 
The Mori cone $\NE(\X)$ is the cone of vectors
\[
\begin{pmatrix}
  l_1\\
  l_2
\end{pmatrix}\in \RR^2
\quad
\text{such that $l_1\geq 0$ and $3l_1+2l_2 \geq 0$.}
\]
We see that $\Lambda^\varnothing \subset \LL_\QQ$ is the lattice of
vectors:
\[
\begin{pmatrix}
  l_1\\
  l_2
\end{pmatrix}
\in \NE \X
\quad
\text{such that $l_1\in \ZZ$ and $l_2 \in \frac1{2} \ZZ$}
\]
and that the reduction function is:
\[
 v^S \colon
\begin{pmatrix}
  l_1\\
  l_2
\end{pmatrix}
\mapsto \langle {-l_2} \rangle
\]
Let us write the element of the Novikov ring corresponding to
$(l_1,l_2) \in \Lambda^\varnothing$ as $Q^{(l_1,l_2)}$.  The
$I$-function (that is, the $S$-extended $I$-function with $S =
\varnothing$) is:
\begin{multline*}
  I(t,x,z) = 
  z e^{(u_1 t_1 + u_2 t_2 + u_3 t_3 + u_4 t_4)/z} \\
  \times
  \sum_{\substack{(l_1,l_2) \in \ZZ \times \frac{1}{2}\ZZ:\\
      l_1 \geq 0, 3l_1+2l_2 \geq 0}}
  \frac{Q^{(l_1,l_2)} e^{(t_1+3t_2+t_3)l_1} e^{(2t_2+t_4)l_2}}
  {\prod_{\substack{\langle b\rangle = 0\\
        0 < b\leq l_1}} (u_1+bz)(u_3+bz)}
 \frac{  \fun_{\langle {-l_2} \rangle}
  }{\prod_{\substack{\langle b\rangle = 0\\
        0 < b\leq 3l_1+2l_2}}(u_2+bz)}
  \frac{\prod_{\substack{\langle b\rangle = \langle l_2 \rangle\\
        b\leq 0}} (u_4+bz)}
  {\prod_{\substack{\langle b\rangle = \langle l_2 \rangle\\
        b\leq l_2}} (u_4+bz)}
\end{multline*}
This is homogeneous of degree $1$ if we set $\deg t_1 = \deg t_2 =
\deg t_3 = \deg t_4 = 0$, $\deg z = 1$, and $\deg Q^{(l_1,l_2)} =
5l_1+3l_2$.  We therefore have:
\[
I(t,x,z) = z \fun_0 + t_1 u_1 \fun_0 + t_2 u_2 \fun_0 + t_3 u_3 \fun_0
+ t_4 u_4 \fun_0 - \textstyle \frac{1}{2} Q^{(1,{-\frac{3}{2}})}
e^{t_1+t_3-\frac{3}{2}t_4} \fun_{\frac{1}{2}} + O(z^{-1})
\]
and condition $\sharp$ fails.

\begin{rmk}
  The coarse moduli space $X$ of $\X$ is the ruled surface ${\mathbb
    F}_3$. Let $A$ and $B$ denote the natural divisors on ${\mathbb
    F}_3$, with $A$ the fibre and $B$ the negative section. Then $\X$
  can be interpreted as the moduli stack of square roots of $B$
  \cite[\S2]{Cadman}, \cite[Appendix~B]{AGV2}. The stack $\X$ contains
  a substack $\{x_4=0\}$ supported on $B$ and isomorphic to
  $\PP(2,2)$. In this context it is natural to identify the
  \emph{integral} Chow group $\Ch (\X,\ZZ)$ with the subring of
  $\Ch^\bullet (X,\QQ)$ multiplicatively generated by $A$ and $B/2$;
  the cycle class of $\PP(2,2)\subset \X$ is $B/2$. This gives an
  interpretation of the degrees $(l_1,l_2)$ occurring in the
  definition of $I(t,x,z)$.
\end{rmk}

\subsection{Example 8: $\PP^2$}
\label{sec:P2_Birkhoff}

There is a well-known closed formula~\cite{Givental:ICM} for the small
$J$-function of $\X = \PP^2$, that is, for the $J$-function
$J_\X(t,z)$ with $t \in H^2(\X)$.  We now show how to use an
$S$-extended $I$-function to obtain arbitrarily many terms of the
Taylor expansion of the big $J$-function of $\X$, that is, of the
$J$-function $J_\X(t,z)$ with $t \in H^\bullet(\X)$.  We use the
Birkhoff factorization procedure described in \cite[\S8]{cg}.  We
will compute the non-equivariant version of the $J$-function, as the
equivariant calculation is significantly more involved.

The variety $\X$ is the toric Deligne--Mumford stack associated to the
stacky fan $\bSigma = (N,\Sigma,\rho)$, where:
\[
\rho = 
\begin{pmatrix}
  -1 & 1 & 0 \\
  -1 & 0 & 1 
\end{pmatrix}
\colon \ZZ^3 \to N=\ZZ^2
\]
and $\Sigma$ is the complete fan in $N_\QQ \cong \QQ^2$ with rays
given by the columns of $\rho$. We have $\Box(\bSigma) = \{0\}$.
Consider the $S$-extended $I$-function where $S = \bigl\{ (0,0), (0,-1)
\bigr\}$ and the map $S \to N_\Sigma$ is the canonical inclusion.  The
$S$-extended fan map is:
\[
\rho^S =
\begin{pmatrix}
  -1 & 1 & 0 & 0 & 0 \\
  -1 & 0 & 1 & 0 & -1
\end{pmatrix}
\colon \ZZ^{3+2} \to N
\]
so that $\LL^S_\QQ\cong \QQ^3$ is identified as a subset of
$\QQ^{3+2}$ via the inclusion:
\[
\begin{pmatrix}
  l  \\
  k_0 \\
  k_1
\end{pmatrix}
\mapsto
\begin{pmatrix}
  1 & 0 & -1 \\
  1 & 0 & -1 \\
  1 & 0 & 0 \\
  0 & 1 & 0 \\
  0 & 0 & 1
\end{pmatrix}
\begin{pmatrix}
  l  \\
  k_0 \\
  k_1
\end{pmatrix}
\] 
The $S$-extended Mori cone is the positive octant. We see that
$\Lambda^S\subset \LL^S_\QQ$ is the lattice of vectors:
\[
\begin{pmatrix}
  l \\
  k_0 \\
  k_1
\end{pmatrix}
\quad \text{such that $l$,~$k_0$,~$k_1 \in \ZZ$}
\]
The reduction function $v^S$ is trivial.  Let $P \in H^2(\X)$ denote
the first Chern class of $\cO(1)$, and identify the Novikov ring
$\Novikov$ with $\CC[\![Q]\!]$ via the map that sends $d \in
H_2(\X;\ZZ)$ to $Q^{\int_d P}$.  The non-equivariant limit of the
$S$-extended $I$-function is:
\[
I^S_\non(t,x,z) = 
z e^{(t_1 + t_2 + t_3)P/z} 
\sum_{(l,k_0,k_1) \in \NN^3}
\frac{Q^l x_0^{k_0} x_1^{k_1} e^{(t_1+t_2+t_3)l}}{z^{k_0+k_1} k_0! k_1!}
\frac
{\prod_{b\leq 0} (P+bz)^2}
{\prod_{b\leq l-k_1} (P+bz)^2}
\frac{1}{\prod_{0<b\leq l}(P+bz)}
\]
This takes values in the non-equivariant cohomology ring
$H^\bullet(\X;\CC) = \CC[P]/(P^3)$.  It is homogeneous of degree $1$
if we set $\deg t_1 = \deg t_2 = \deg t_3 = 0$, $\deg x_0 = \deg z =
1$, $\deg x_1 = {-1}$, and $\deg Q = 3$.  
Note that, unlike the other examples in this paper, in this case the $I$-function contains
arbitrarily large positive powers of $z$; this reflects the fact that
some of the variables have negative degree.

We have:
\[
I^S_\non(t,x,z) = \textstyle z +\frac{1}{2} z x_1^2 P^2+ 
x_0 + (t_1+t_2+t_3) P + x_1 P^2 
 + \frac{1}{2} x_0 
x_1^2 P^2 + O(z^{-1}) + O(x_1^3) 
\]
The non-equivariant version of the mirror theorem for toric
Deligne--Mumford stacks \cite[Corollary~32]{ccit2} gives that
$I^S(t,x,-z) \in \sL^\non$, where $\sL^\non$ is the Givental
cone for non-equivariant Gromov--Witten theory (see e.g.~\cite[\S
3]{ccit}).  Set $t_2 = t_3 = 0$.  Condition $S$-$\sharp$ holds modulo
$x_1^2$, so:
\[
J_\X(x_0 + t_1 P + x_1 P^2,z) + O(x_1^2) = I^S_\non(t,x,z) + O(x_1^2)
\]
The following elements lie in $T_{I^S_\non(t,x,z)} \sL^\non$:
\begin{align*}
  & \frac{\partial I^S_\non}{\partial x_0} = e^{Pt_1/z} e^{x_0/z} 
\Big( 1 + O(z^{-2}) + O(x_1) \Big) \\
  & \frac{\partial I^S_\non}{\partial t_1} = e^{Pt_1/z} e^{x_0/z} 
\Big( P 
   + O(z^{-2}) + O(x_1) \Big) \\
  & \frac{\partial I^S_\non}{\partial x_1} = 
  \textstyle e^{Pt_1/z} e^{x_0/z} 
\Big( P^2 + z^{-1} Q e^{t_1}  
  +O(z^{-2}) + O(x_1) \Big)
\end{align*} 
We have that:
\[
I^S_\non(t,x,z) = \textstyle e^{P t_1/z} e^{x_0/z} 
\Big(z + \frac{1}{2} z x_1^2 P^2 + x_1 P^2 
 + O(z^{-1}) +  O(x_1^3)  \Big)
\]
and general properties of $\sL^\non$ guarantee \cite{gi2}
\cite[Appendix~B]{ccit2} that:
\begin{multline*}
  I^S_\non(t,x,-z) 
  + C_0(t,x,z) z \frac{\partial I^S_\non}{\partial x_0}(t,x,{-z}) \\
  + C_1(t,x,z) z \frac{\partial I^S_\non}{\partial t_1}(t,x,{-z})
  + C_2(t,x,z) z \frac{\partial I^S_\non}{\partial x_1}(t,x,{-z})
  \in \sL^\non
\end{multline*}
for any $C_0$,~$C_1$,~$C_2$ depending polynomially on $z$.  
But:
\begin{align*}
  I^S_\non(t,x,z) - {\textstyle \frac{1}{2}} z x_1^2 
\frac{\partial I^S_\non}{\partial x_1}
  & = 
  \textstyle e^{P t_1/z} e^{x_0/z} \Big(z + x_1 P^2 - \frac{1}{2}
  x_1^2 Q e^{t_1} + 
  O(z^{-1})  + O(x_1^3) \Big)
  \\
  &= z + x_0  - \textstyle \frac{1}{2} x_1^2 Q e^{t_1} 
+ t_1 P + x_1 P^2 + O(z^{-1}) + O(x_1^3) 
\end{align*}
and thus:
\[
J_\X(\tau,z) + O(x_1^3) = 
I^S_\non(t,x,z) - {\textstyle \frac{1}{2}} x_1^2 z \frac{\partial I^S_\non}{\partial
  x_1}+ O(x_1^3) 
\]
where:
\[
\tau (x_0,t_1,x_1) = \big(x_0  - \textstyle \frac{1}{2} x_1^2 Q e^{t_1}\big)1 
+ t_1 P + x_1 P^2 + O(x_1^3)
\]
Inverting the mirror map $(x_0,t_1,x_1) \mapsto \tau$ gives a
closed-form expression for the big $J$-function $J_\X(a_0 + a_1 P +
a_2 P^2,z)$ to order $2$ in $a_2$.  One can repeat this procedure to
compute the big $J$-function $J_\X(a_0 + a_1 P + a_2 P^2,z)$ to
arbitrarily high order in $a_2$. 

\begin{rmk} 
The extended $I$-function is closely related to Barannikov's 
big quantum cohomology mirror for $\PP^n$ \cite{Barannikov:semiinf}: 
it satisfies the Picard--Fuchs differential equation for the mirror oscillatory 
integrals \cite[Example 4.15]{Iritani:QDM}. 
``Big'' mirror symmetry for $\PP^2$ has been also studied via 
tropical geometry~\cite{Gross} and via quasimap theory~\cite{Jinzenji-Shimizu,CF-Kim:bigI}. 
\end{rmk} 

\section{Twisted $I$-Functions}
\label{sec:twisted_I_functions}

Let $\X$ be the toric Deligne--Mumford stack defined by an
$S$-extended stacky fan $\bSigma$, as in \S\ref{sec:extended_stacky_fans}. 
Suppose that the coarse moduli space of $\X$ is semi-projective.  
Let $\varepsilon_1,\ldots,\varepsilon_r \in (\LL^S)^\vee$.  The canonical
inclusion $i\colon \LL \to \LL^S$ induces $i^\vee \colon (\LL^S)^\vee
\to \LL^\vee = \Pic(\X)$, and so the classes
$\varepsilon_1,\ldots,\varepsilon_r$ define line bundles
$\cE_1,\ldots,\cE_r$ over $\X$ via $i^\vee$.  Let $\cE = \cE_1 \oplus
\cdots \oplus \cE_r$, and let $\bc$ denote the invertible
multiplicative characteristic class
\begin{equation}
\label{eq:bc}
\bc({-}) = \exp\bigl(\textstyle \sum_{k=0}^\infty s_k \ch_k({-})\bigr)
\end{equation} 
where $s_0,s_1,\ldots$ are parameters.  We consider the Gromov--Witten
theory of $\X$ twisted, in the sense of \cite{cg,ts,ccit}, by the
vector bundle $\cE$ and the characteristic class $\bc$.  Let $\sL^\tw$
denote Givental's Lagrangian cone for $(\bc,\cE)$-twisted
Gromov--Witten theory, as in \cite[\S3]{ccit}.

Let $D_i \in H^2(\X;\CC)$, denote the (non-equivariant) class
Poincar\'e dual to the $i$th toric divisor for $1 \leq i \leq n$, and
the zero class for $n < i \leq n+|S|$.  Let $E_j \in H^2(\X;\CC)$, $1
\leq j \leq r$, denote the (non-equivariant) first Chern class of
$\cE_j$.  Let $\sL^\non$ denote Givental's Lagrangian cone for the
non-equivariant Gromov--Witten theory of $\X$, as in \cite[\S3]{ccit},
and let $I_\non^S(t,x,z)$ denote the non-equivariant limit of the
$S$-extended $\TT$-equivariant $I$-function $I^S(t,x,z)$. 

\begin{notation}
  Given parameters $s_0,s_1,s_2,\ldots$ as above, write $s(x) :=
  \sum_{k=0}^\infty s_k \frac{x^k}{k!}$.
\end{notation}

\begin{defn}[\protect{cf.~\cite[\S4]{ccit}}]
  Given $b \in \Box(\bSigma)$ and $\lambda \in \Lambda E^S_b$, define
  the \emph{modification factor}:
  \[
  M_{\lambda,b}(z) := \prod_{j=1}^r 
  \frac{
    \prod_{\substack{a : \langle a \rangle
        = \langle \varepsilon_j\cdot \lambda \rangle \\ 
   a \leq \varepsilon_j \cdot\lambda}} \exp\big( s(E_j + a z)\big)
    }
    {
      \prod_{\substack{a : \langle a \rangle
          = \langle \varepsilon_j \cdot \lambda \rangle \\a \leq 0}} \exp\big( s(E_j + a z) \big)
    }      
    \]
\end{defn}

\begin{defn}
  The \emph{$S$-extended $(\bc,\cE)$-twisted $I$-function of $\X$} is:
  \[
  I^S_{\bc,\cE}(t,x,z) = \sum_{b \in \Box(\bSigma)}
\sum_{\lambda \in \Lambda E^S_b} \tQ^\lambda
  I_{\lambda,b}(z) M_{\lambda,b}(z) y^b
  \]
  where $I_\non^S(t,x,z) = \sum_{b \in \Box(\bSigma)} \sum_{\lambda \in \Lambda E^S_b} \tQ^\lambda
  I_{\lambda,b}(z) y^b$.
\end{defn}

\begin{thm}
  \label{thm:twisted_I_is_on_the_twisted_cone}  
  With hypotheses and notation as above, we have that
  $I^S_{\bc,\cE}(t,x,-z)\in \sL^\tw$.
\end{thm}

\begin{proof}
  Recall from the proof of Theorem~4.8 in \cite{ccit} that:
  \[
  G_y(x,z) := \sum_{l=0}^\infty \sum_{m=0}^\infty 
s_{m+l-1} \frac{B_m(y)}{m!}
  \frac{x^l}{l!} z^{m-1}
  \]
  satisfies:
  \begin{equation}
    \label{eq:G_identities}
    \begin{aligned}
      & G_y(x,z) = G_0(x+yz,z) \\
      & G_0(x+z,z) = G_0(x,z) + s(x)
    \end{aligned}
  \end{equation}
  Here $B_m(y)$ is the $m$th Bernoulli polynomial, and $s_{-1}$ is
  defined to be zero.  Thus:
  \begin{align*}
    M_{\lambda,b}(-z) & = 
    \prod_{j=1}^r
    \exp\Bigg(
    \sum_{\substack{a : \<a\> = \<\epsilon_j \cdot \lambda\> \\  a \leq \epsilon_j\cdot
      \lambda}}
    s(E_j - a z) 
    -
    \sum_{\substack{a : \<a\> = \<\epsilon_j \cdot \lambda\> \\  a \leq 0}}
    s(E_j - a z) 
    \Bigg) \\
    & = \prod_{j=1}^r \exp\Big(
    G_0\big(E_j + \<{-\varepsilon_j\cdot \lambda}\>z,z\big)
    -
    G_0\big(E_j - (\varepsilon_j\cdot\lambda)z,z\big)
    \Big) \\
    &= \prod_{j=1}^r
    \exp\Big(
    G_{\<{-\varepsilon_j\cdot \lambda}\>}(E_j,z)
    -
    G_0\big (E_j - (\varepsilon_j\cdot\lambda)z,z\big)
    \Big) 
  \end{align*}
  where for the last two equalities we used \eqref{eq:G_identities}.
  
  Let $b \in \Box(\bSigma)$, and let $f(b,j) \in [0,1)$ be the
  rational number such that if $(x,g) \in I\X_b$, then $g$ acts on the
  fiber of $\cE_j$ over $x \in \X$ by multiplication by $\exp\big(2
  \pi \sqrt{-1} f(b,j)\big)$.  
  Note that if $\lambda \in \Lambda E^S_b$ then $f(b,j) =
  \<{-\varepsilon_j\cdot\lambda}\>$. 
  Tseng has proven \cite{ts} that the operators: 
  \[
  \Delta_j := \bigoplus_{b \in \Box(\bSigma)}
  \exp\Big(G_{f(b,j)}(E_j,z)\Big)
  \]
  and $\Delta := \prod_{j=1}^r \Delta_j$ satisfy $\Delta(\sL^\non) =
  \sL^\tw$; the direct sum in the definition of $\Delta_j$ here
  reflects the decomposition $H^\bullet_\CR(\X;\CC) = \bigoplus_{b \in
    \Box(\bSigma)} H^{\bullet - \age(b)}(I\X_b;\CC)$.  To show that
  $I^S_{\bc,\cE}(t,x,-z)$ lies in $\sL^\tw$, therefore, it suffices to
  show that:
  \begin{equation}
    \label{eq:look_at_this}
    \sum_{b \in \Box(\bSigma)} \sum_{\lambda \in \Lambda E^S_b}
    \tQ^\lambda I_{\lambda,b}
    \prod_{j=1}^r
    \exp\Big( {- G_0\big (E_j - (\varepsilon_j\cdot\lambda)z,z\big)} \Big)
    y^b
  \end{equation}
  lies in the cone $\sL^\non$ for the untwisted theory.  
Recall the splitting $\LL^S\otimes \QQ \cong (\LL \otimes \QQ) \oplus \QQ^m$ 
from \eqref{eq:LS_splitting}. 
Under this splitting, $\varepsilon_j\in (\LL^S)^\vee$ induces 
$E_j \in \LL^\vee \otimes \QQ \cong H_2(\X;\QQ)$ and 
$(f_{j1},\dots,f_{jm}) \in \QQ^m$. 
Choose $e_{ji}\in \QQ$ such that $E_j = \sum_{i=1}^n e_{ji} D_i$ 
and define the differential operator $\nabla_{\varepsilon_j}$ by 
\[
\nabla_{\varepsilon_j} = e_{j1} \frac{\partial}{\partial t_1} + 
\cdots + e_{j n} \frac{\partial}{\partial t_n} + 
f_{j1} x_1 \frac{\partial}{\partial x_1} + \cdots 
+ f_{jm} x_m \frac{\partial}{\partial x_m} 
\]
Then 
  \eqref{eq:look_at_this} is:
  \begin{equation}
    \label{eq:its_equal_to_this}
    \prod_{j=1}^r
    \exp\Big( {- G_0\big ({-z} \nabla_{\varepsilon_j},z\big)} \Big)
    I_\non^S(t,x,{-z})
  \end{equation}
and we know by the non-equivariant version of
  the mirror theorem for toric Deligne--Mumford stacks
  \cite[Corollary~32]{ccit2}, \cite{ccfk} that $I^S_\non(t,x,{-z}) \in
  \sL^\non$. Arguing as in the proof of \cite[Theorem~4.8]{ccit} now
  shows that:
  \[
  \prod_{j=1}^r
  \exp\Big( {- G_0\big ({-z} \nabla_{\varepsilon_j},z\big)} \Big)
  I_\non^S(t,x,{-z}) \in \sL^\non
  \]
  as required.
\end{proof}

\begin{rmk}
  Theorem~\ref{thm:twisted_I_is_on_the_twisted_cone}, roughly
  speaking, states that a certain hypergeometric modification of the
  untwisted $I$-function $I^S_\non$ lies on the twisted cone
  $\sL^\tw$.  The proof of
  Theorem~\ref{thm:twisted_I_is_on_the_twisted_cone} is essentially
  the same as the proof of Theorem~4.8 in \cite{ccit}, where we showed
  that a hypergeometric modification of the untwisted $J$-function
  $J_\X$ lies on $\sL^\tw$.  The essential properties of the
  $J$-function $J_\X$ used there are that $J_\X(t,{-z}) \in \sL^\non$
  (which holds by definition) and the Divisor Equation
  \cite[Lemma~4.7(3)]{cclt}.  The essential properties of the
  $I$-function $I^S_\non$ used here are that $I^S_\non(x,t,{-z}) \in
  \sL^\non$ (our mirror theorem) and that $\nabla_{\varepsilon_j}
  I^S_\non(t,x,z) = \big(E_j + (\varepsilon_j\cdot\lambda)z\big)
  I^S_\non(x,t,z)$.  This latter property, which is a version of the
  Divisor Equation for the $I$-function, allows us to replace
  \eqref{eq:look_at_this} by \eqref{eq:its_equal_to_this}.
\end{rmk}

\begin{rmk}
  With a little extra effort --- modifying the formal setup in
  \cite{ccit} to include equivariant parameters --- one could prove
  the $\TT$-equivariant analog of
  Theorem~\ref{thm:twisted_I_is_on_the_twisted_cone}
  in exactly the same way.  We omit this
  here, however, as we know of no applications of these results.  In
  current applications one either treats toric complete intersections
  by taking $\bc$ to be the $S^1$-equivariant Euler class $\be$ 
(see \S \ref{sec:ci_examples}), 
  in which case the $\TT$-action on the ambient
  space is irrelevant as it does not preserve the complete
  intersection, or one treats non-compact geometries by taking $\bc =
  \be^{-1}$ to be the $S^1$-equivariant inverse Euler class.  The
  latter case can be treated directly using
  Theorem~\ref{thm:toric_mirror_theorem}, as the total space of $\cE$
  is itself a toric Deligne--Mumford stack.
\end{rmk}

\section{A Mirror Theorem for Toric Complete Intersection Stacks}
\label{sec:ci_examples}

We now describe how, under appropriate hypotheses on the vector bundle
$\cE = \cE_1 \oplus \cdots \oplus \cE_r$ 
and the toric Deligne--Mumford stack $\X$, 
Theorem \ref{thm:twisted_I_is_on_the_twisted_cone} gives 
a mirror theorem for the complete 
intersection stack $\Y$ cut out by a generic section of $\cE$.
Suppose that the vector bundle $\cE$ is \emph{convex}, that is, for 
every genus-zero stable map $f\colon (\cC,x_1,\dots,x_n) \to \X$ 
from a pointed orbicurve $(\cC,x_1,\dots,x_n)$, 
one has $H^1(\cC,f^*\cE) = 0$. 
This is a very restrictive assumption: it is equivalent to requiring that:
\begin{description} 
\item[(positivity)] $c_1(\cE_j) \cdot d\ge 0$ 
for every degree $d$ of a genus-zero stable map; and

\item[(coarseness)] $\cE_j$ is the pull-back of a line bundle on 
the coarse moduli space of $\X$;
\end{description} 
hold for all $j=1,\dots,r$. 
See \cite{qlhp} for a detailed discussion of convexity for vector
bundles on orbifolds. 
Under these conditions, we may choose $\varepsilon_j \in (\LL^S)^\vee$ 
in \S \ref{sec:twisted_I_functions} so that it vanishes on the 
vectors in \eqref{eq:splitting_vector}, i.e.~$\varepsilon_j$ corresponds to 
$(\text{the class of $\cE_j$},0)$ under the splitting $(\LL^S)^\vee \otimes \QQ 
\cong (\LL^\vee \otimes \QQ) \oplus \QQ^m = (\Pic(\X)\otimes\QQ) 
\oplus \QQ^m$ induced by \eqref{eq:LS_splitting}. 
In fact, the coarseness implies that the pairings of $\varepsilon_j$ 
with the vectors in \eqref{eq:splitting_vector} lie in $\ZZ$, 
and in view of the exact sequence: 
\[
\begin{CD} 
0@>>> (\ZZ^m)^* @>>> (\LL^S)^\vee @>>> \LL^\vee @>>> 0
\end{CD}
\]
one may always shift $\varepsilon_j$ by the action of $(\ZZ^m)^*$ 
so that these pairings vanish. 
We take the characteristic class $\bc$ in \eqref{eq:bc} to be 
the $S^1$-equivariant Euler class $\be$: 
\[
\be(\cV) = \prod_{\text{$v$: Chern roots of $\cV$}} (\kappa + v)
\]
where we consider the fiberwise $S^1$-action on 
vector bundles and $\kappa$ denotes the $S^1$-equivariant 
parameter.  
This corresponds to the choice of parameters 
\begin{align*} 
s_k &= \begin{cases} 
\log \kappa & \text{for $k=0$} \\  
(-1)^{k-1} (k-1)! \kappa^{-k} & \text{for $k\ge 1$} 
\end{cases} 
\end{align*} 
With the convexity hypothesis for $\cE$ 
and the choices for $\varepsilon_j$ as above,
the $(\be,\cE)$-twisted $I$-function takes the form: 
\begin{multline*} 
I^S_{\be,\cE}(t,x,z) = z e^{\sum_{i=1}^n t_i D_i/z} \times \\
\sum_{b\in \Box(\bSigma)} 
\sum_{\lambda \in \Lambda E^S_b} 
\tQ^\lambda e^{\lambda t} 
\left( 
      \prod_{i=1}^{n+m}
      \frac{\prod_{\<a\>=\< \lambda_i \>, a \leq 0}(D_i+a z)}
      {\prod_{\<a \>=\< \lambda_i \>, a\leq \lambda_i} (D_i+a z)} 
\right) 
\left(
\prod_{i=1}^r 
\prod_{a=1}^{E_j \cdot d} 
(\kappa + E_j + a z) 
\right)
    y^b  
\end{multline*}
where we write $\lambda = (d,k)$ via \eqref{eq:LS_splitting}. 
Note that $E_j\cdot d \in \ZZ_{\ge 0}$ by the convexity 
assumption. 
Let $i^\star\colon H_\CR^\bullet(\X) \to H_\CR^\bullet(\Y)$ 
denote the pullback along the inclusion $i\colon \Y \to \X$, 
and define:
\[
I^S_{\Y}(t,x,z) = \lim_{\kappa \to 0}i^\star I^S_{\be,\cE}(t,x,z)
\]
Theorem~\ref{thm:twisted_I_is_on_the_twisted_cone} 
implies that $I^S_{\be,\cE}(t,x,-z)$ lies in the 
$(\be,\cE)$-twisted cone $\sL^\tw$.  Combining this with functoriality for the virtual fundamental class~\cite[\S2]{Pandharipande:afterGivental}, \cite{kkp} either as in the argument of~\cite[Proposition~2.4]{Iritani:periods} or using \cite[Theorem~1.1 and Remark~2.2]{Coates:cone}, proves the following Mirror Theorem for complete intersections in toric Deligne--Mumford stacks.

\begin{thm} 
\label{thm:ci_mirror} 
Let $\X$ be a toric Deligne--Mumford stack with semi-projective 
coarse moduli space, and let $S$ be a finite set 
equipped with a map $S \to N_\Sigma$ 
as in Theorem \ref{thm:toric_mirror_theorem}.  
Let $\cE = \cE_1 \oplus \cdots \oplus \cE_r$ be the sum 
of convex line bundles over $\X$ 
and $\Y \subset \X$ be the zero-locus of a transverse section of $\cE$. 
Then $I_\Y^S(t,x,-z)$ lies in Givental's Lagrangian submanifold $\sL_\Y$ for $\Y$. 
\end{thm} 

\begin{rmk} 
We chose $\varepsilon_j$ so that it vanishes on the vectors 
\eqref{eq:splitting_vector}: this choice yields the optimal 
$z^{-1}$-asymptotics for $I_\Y^S$. 
In order for $I_{\be,\cE}^S$ 
to have a well-defined non-equivariant limit $\kappa \to 0$, 
we only need to assume that $\varepsilon_j$ pairs with the vectors \eqref{eq:splitting_vector} 
non-negatively, and Theorem \ref{thm:ci_mirror} is still valid under this 
weaker assumption. 
When the pairings of $\varepsilon_j$ with the vectors \eqref{eq:splitting_vector} 
are positive, the corresponding $I$-function $I^S_\Y$ has worse $z^{-1}$-asymptotics, 
i.e.~contains higher powers in $z$. 
\end{rmk} 

\begin{rmk} 
  Theorem~\ref{thm:ci_mirror} can also be proved by combining the methods of
  Cheong--Ciocan-Fontanine--Kim~\cite{ccfk} with the methods of~\cite[\S7]{cfk}.  
  This gives a different approach, which relies on virtual localization rather than the quantum Lefschetz theorem.
\end{rmk}

Suppose that the $I$-function $I^S_{\Y}$ has the following 
asymptotics (cf.~Condition $S$-$\sharp$ in \S \ref{sec:sharp})
\begin{equation}
  \label{eq:good_asymptotics}
  I^S_{\Y}(t,x,z) = F(t,x) z + G(t,x) + O(z^{-1})
\end{equation}
where $F$ is an $H^0(\Y)$-valued function 
and $G$ is an $H_{\CR}^\bullet(\Y)$-valued function. 
This holds, for example, if the following conditions are met: 
\begin{itemize} 
\item $c_1(T\X) - c_1(\cE)$ is nef, and 
\item the image of 
$S\to N_\Sigma$ is contained in 
$\{b \in N_\Sigma: \age(b)\le 1\}$. 
\end{itemize} 
Define the mirror map by:
\[
\tau(t,x) = \frac{G(t,x)}{F(t,x)}
\]
Theorem~\ref{thm:ci_mirror} determines the unique point 
$F(t,x)^{-1} I^S_{\Y}(t,x,{-z})$ on $\sL_\Y$ of the form 
${-z} + \tau + O(z^{-1})$: this is the $J$-function $J_\Y(\tau,-z)$ for $\Y$. 
Thus we obtain the following mirror theorem. 
\begin{cor} 
\label{cor:ci_mirrorthm}
With hypotheses and notation as above, we have:
\[
J_{\Y}(\tau(t,x),z) = \frac{I_\Y^S(t,x,z) }{F(t,x)} 
\]
\end{cor} 

\begin{rmk}
  In general the $(\be,\cE)$-twisted $I$-function will not satisfy
  \eqref{eq:good_asymptotics}, but one can still obtain the
  $(\be,\cE)$-twisted $J$-function by Birkhoff factorization as in
  \S\ref{sec:P2_Birkhoff}.  Provided that the bundle $\cE$ is convex,
  this allows the computation of genus-zero Gromov--Witten invariants
  of $\Y$.
\end{rmk}

\begin{rmk}
  If $\cE$ is not convex then the relationship between
  $(\be,\cE)$-twisted Gromov--Witten invariants of $\X$ and
  Gromov--Witten invariants of $\Y$ is not well understood
  \cite{qlhp}.  This merits further investigation.
\end{rmk}

\begin{rmk} Certain components 
of $I_\Y^S$ can be written as (exponential) periods of 
the Landau--Ginzburg model mirror to $\Y$: see \cite{Iritani:periods}.  
\end{rmk}

\subsection{Example 9: a sextic hypersurface in $\PP(1,1,1,3,3)$}

Let the orbifold $\Y$ be a smooth sextic hypersurface in $\X =
\PP(1,1,1,3,3)$; this is a Fano \mbox{$3$-fold} with canonical
singularities.  The ambient space $\X$ is the toric Deligne--Mumford
stack associated to the stacky fan $\bSigma = (N,\Sigma,\rho)$, where:
\[
\rho = 
\begin{pmatrix}
  -1 & 1  & 0 & 0 & 0 \\
  -1 & 0 & 1 & 0 & 0 \\
  -3 & 0 & 0 & 1 & 0 \\
  -3 & 0 & 0 & 0 & 1
\end{pmatrix}
\colon \ZZ^5 \to N=\ZZ^4
\]
and $\Sigma$ is the complete fan in $N_\QQ \cong \QQ^4$ with rays
given by the columns $\rho_1,\ldots,\rho_5$ of $\rho$.  We identify
$\Box(\bSigma)$ with the set $\bigl\{0,\frac{1}{3},\frac{2}{3}\bigr\}$
via the map $\kappa \colon x \mapsto x(\rho_1+\rho_2+\rho_3)$.
Consider the $S$-extended $I$-function where
$S=\bigl\{0,\frac{1}{3}\bigr\}$ and $S \to N_\Sigma$ is the map
$\kappa$.  The $S$-extended fan map is:
\[
\rho^S =
\begin{pmatrix}
  -1 & 1  & 0 & 0 & 0 & 0 & 0 \\
  -1 & 0 & 1 & 0 & 0 & 0 & 0\\
  -3 & 0 & 0 & 1 & 0 & 0 & -1\\
  -3 & 0 & 0 & 0 & 1 & 0 & -1
\end{pmatrix}
\colon \ZZ^{5+2} \to N
\]
so that $\LL^S_\QQ\cong \QQ^3$ is identified as a subset of
$\QQ^{5+2}$ via the inclusion:
\[
\begin{pmatrix}
  l  \\
  k_0\\
  k_1
\end{pmatrix}
\mapsto
\begin{pmatrix}
  \frac{1}{3} & 0 & {-\frac{1}{3}} \\
  \frac{1}{3} & 0 & {-\frac{1}{3}} \\
  \frac{1}{3} & 0 & {-\frac{1}{3}} \\
  1 & 0 & 0 \\
  1 & 0 & 0 \\
  0 & 1 & 0 \\
  0 & 0 & 1
\end{pmatrix}
\begin{pmatrix}
  l  \\
  k_0\\
  k_1\\
\end{pmatrix}
\]
The $S$-extended Mori cone is the positive octant.  We see that
$\Lambda^S \subset \LL^S_\QQ$ is the sublattice of vectors:
\[
\begin{pmatrix}
  l\\
  k_0\\
  k_1
\end{pmatrix}
\quad \text{such that $l$,~$k_0$,~$k_1 \in \ZZ$}
\]
and that the reduction function is:
\[
v^S\colon 
\begin{pmatrix}
  l\\
  k_0\\
  k_1\\
\end{pmatrix}
\mapsto
 \Bigl\langle \frac{k_1-l}{3}\Bigr\rangle
\]
Let $P \in H^2(\X;\QQ)$ denote the (non-equivariant) first Chern class
of $\cO_\X(1)$, and identify the Novikov ring $\Novikov$ with
$\CC[\![Q]\!]$ via the map that sends $d \in H_2(X;\ZZ)$ to $Q^{\int_d
  3P}$.  With notation as in \S\ref{sec:twisted_I_functions} we have
$D_1 = D_2 = D_3 = P$, $D_4 = D_5 = 3P$, and so the non-equivariant
limit of the $S$-extended $I$-function is:
\begin{multline*}
  I^S_\non(t,x,z) = 
  z e^{(t_1 + t_2 + t_3 + 3t_4 + 3t_5)P/z}  \\
  \times
  \sum_{(l,k_0,k_1) \in \NN^3}
  \frac{Q^l x_0^{k_0} x_1^{k_1} e^{(t_1+t_2 + t_3 + 3t_4 + 3t_5)l}}
  {z^{k_0+k_1} k_0! k_1!}
  \frac{
    \prod_{\begin{subarray}{l}
       \langle b\rangle = \<\frac{l-k_1}{3}\> \\ 
       \ b\leq 0 \end{subarray}}
    (P+bz)^3
  }{
    \prod_{
    \begin{subarray}{l}
       \langle b\rangle = \<\frac{l-k_1}{3}\> \\ 
       \ b\leq \frac{l-k_1}{3} 
    \end{subarray}}
    (P+bz)^3
  }
  \frac{
    \fun_{\bigl\langle \frac{k_1-l}{3}\bigr \rangle}
  }{
    \prod_{\substack{\langle b\rangle = 0 \\ 
        1 \leq b\leq l}}
    (3P+bz)^2
  }
\end{multline*}

Let $\cE \to \X$ be the line bundle corresponding to the element
$\varepsilon \in (\LL^S)^\vee$ given by:
\[
\varepsilon \colon
\begin{pmatrix}
  l \\
  k_0 \\
  k_1
\end{pmatrix}
\mapsto 2l
\]
so that $\cE = \cO(6)$.  The $S$-extended $(\be,\cE)$-twisted
$I$-function of $\X$ is:
\begin{multline*}
  I^S_{\be,\cE}(t,x,z) = 
  z e^{(t_1 + t_2 + t_3 + 3t_4 + 3t_5)P/z}  \\
  \times
  \sum_{(l,k_0,k_1) \in \NN^3}
  \frac{Q^l x_0^{k_0} x_1^{k_1} e^{(t_1+t_2 + t_3 + 3t_4 + 3t_5)l}}
  {z^{k_0+k_1} k_0! k_1!}
  \frac{
    \prod_{
    \begin{subarray}{l} 
    \langle b \rangle = \<\frac{l-k_1}{3}\> \\ 
    \ b \leq 0
    \end{subarray}
    }
    (P+bz)^3
  }{
    \prod_{
    \begin{subarray}{l}  
    \langle b\rangle = \< \frac{l-k_1}{3} \> \\ 
    \ b \leq \frac{l-k_1}{3}
    \end{subarray} 
    }
    (P+bz)^3
  }
  \frac{
    \prod_{\substack{\langle b\rangle = 0 \\ 
        1 \leq b\leq 2l}}
    (\kappa + 6P+bz)
  }{
    \prod_{\substack{\langle b\rangle = 0 \\ 
        1 \leq b\leq l}}
    (3P+bz)^2
  }
  \fun_{\<\frac{k_1-l}{3}\>}
\end{multline*}
This is homogeneous of degree $1$ if we set $\deg t_1 = \deg t_2 =
\deg t_3 = \deg t_4 = \deg t_5 = 0$, $\deg z = \deg Q = \deg x_0 =
\deg \kappa = 1$, and $\deg x_1 = 0$.  We therefore have:
\[
I^S_{\Y}(t,x,z) = z + (t_1+t_2+t_3+3t_4+3t_5)P +  x_0 \fun_0 +
f(x_1) \fun_{\frac{1}{3}} + O(z^{-1})
\]
where:
\[
f(x) = \sum_{m=0}^\infty (-1)^m \frac{x^{3m+1}}{(3m+1)!}
\frac{\Gamma(m + \frac{1}{3})^3}{\Gamma(\frac{1}{3})^3}
\]
Let $g$ denote the power series inverse to $f$, so that $g(x) = x +
\frac{x^4}{648} + \cdots$, and set:
\begin{align*}
  t_i = 
  \begin{cases}
    \tau & \text{if $i=1$} \\
    0 & \text{otherwise}
  \end{cases}
  &&
  x_i = 
  \begin{cases}
    \xi_0 & \text{if $i=0$} \\
    g(\xi_1) & \text{if $i=1$}
  \end{cases}
\end{align*}
Then:
\[
I^S_{\Y}(t,x,z) = 
z + \tau P + \xi_0 \fun_0 + \xi_1 \fun_{\frac{1}{3}} + O(z^{-1})
\]
and Corollary~\ref{cor:ci_mirrorthm} implies that:
\begin{multline*}
  J_\Y\big(
  \tau P + \xi_0 \fun_0 + \xi_1
  \fun_{\frac{1}{3}},z\big) =
  \\
  z e^{\tau P/z}  
  \sum_{(l,k_0,k_1) \in \NN^3}
  \frac{Q^l \xi_0^{k_0} g(\xi_1)^{k_1} e^{\tau l}}
  {z^{k_0+k_1} k_0! k_1!}
  \frac{
    \prod_{
    \begin{subarray}{l}
    \langle b \rangle = \<\frac{l-k_1}{3}\> \\ 
    \ b\leq 0
    \end{subarray}
    }
    (P+bz)^3
  }{
    \prod_{
    \begin{subarray}{l} 
    \langle b \rangle = \<\frac{l-k_1}{3}\> \\ 
    \ b\leq \frac{l-k_1}{3} 
    \end{subarray}
    }
    (P+bz)^3
  }
  \frac{
    \prod_{\substack{\langle b\rangle = 0 \\ 
        0 \leq b\leq 2l}}
    (6P+bz)
  }{
    \prod_{\substack{\langle b\rangle = 0 \\ 
        1 \leq b\leq l}}
    (3P+bz)^2
  }
 \fun_{\<\frac{k_1-l}{3}\>}
\end{multline*}
For example, the coefficient of 
$\fun_0$ in $J_\Y\big( \tau P + \xi_0 \fun_0 + \xi_1
\fun_{\frac{1}{3}},z\big)$ is:
\[
\sum_{l=0}^\infty
\sum_{k_0=0}^\infty
\sum_{\substack{k_1: 0 \leq k_1 \leq l \\
    k_1 \equiv l \bmod 3}}
\frac{Q^l \xi_0^{k_0} g(\xi_1)^{k_1} e^{\tau l}}
{z^{k_0+k_1-1} k_0! k_1!}
\frac{1}{\bigl(\frac{l-k_1}{3}\bigr)!}
\frac{(2l)!}{(l!)^2}
\]
This is the so-called \emph{quantum period} of $\Y$.

\bibliographystyle{plain}
\bibliography{bibliography}

\begin{thebibliography}{10}

\bibitem{AGV2}
Dan Abramovich, Tom Graber, and Angelo Vistoli.
\newblock Gromov-{W}itten theory of {D}eligne-{M}umford stacks.
\newblock {\em Amer. J. Math.}, 130(5):1337--1398, 2008.

\bibitem{Barannikov:semiinf}
Serguei Barannikov.
\newblock Semi-infinite {H}odge structures and mirror symmetry for projective
  spaces.
\newblock \href{http://arxiv.org/abs/math/0010157}{\texttt{arXiv:math/0010157
  [math.AG]}}, 2000.

\bibitem{Barannikov:quantum}
Serguei Barannikov.
\newblock Quantum periods. {I}. {S}emi-infinite variations of {H}odge
  structures.
\newblock {\em Internat. Math. Res. Notices}, (23):1243--1264, 2001.

\bibitem{BCS}
Lev~A. Borisov, Linda Chen, and Gregory~G. Smith.
\newblock The orbifold {C}how ring of toric {D}eligne-{M}umford stacks.
\newblock {\em J. Amer. Math. Soc.}, 18(1):193--215 (electronic), 2005.

\bibitem{Cadman}
Charles Cadman.
\newblock Using stacks to impose tangency conditions on curves.
\newblock {\em Amer. J. Math.}, 129(2):405--427, 2007.

\bibitem{CR2}
Weimin Chen and Yongbin Ruan.
\newblock Orbifold {G}romov-{W}itten theory.
\newblock In {\em Orbifolds in mathematics and physics ({M}adison, {WI},
  2001)}, volume 310 of {\em Contemp. Math.}, pages 25--85. Amer. Math. Soc.,
  Providence, RI, 2002.

\bibitem{CR1}
Weimin Chen and Yongbin Ruan.
\newblock A new cohomology theory of orbifold.
\newblock {\em Comm. Math. Phys.}, 248(1):1--31, 2004.

\bibitem{ccfk}
Daewoong Cheong, Ionut Ciocan-Fontanine, and Bumsig Kim.
\newblock Orbifold quasimap theory.
\newblock \href{http://arxiv.org/abs/1405.7160}{\texttt{arXiv:1405.7160
  [math.AG]}}, 2014.

\bibitem{cfk}
Ionut Ciocan-Fontanine and Bumsig Kim.
\newblock Wall-crossing in genus zero quasimap theory and mirror maps.
\newblock \href{http://arxiv.org/abs/1304.7056}{\texttt{arXiv:1304.7056
  [math.AG]}}, 2013.

\bibitem{CF-Kim:bigI}
Ionut Ciocan-Fontanine and Bumsig Kim.
\newblock Big {$I$}-functions.
\newblock \href{http://arxiv.org/abs/1401.7417}{\texttt{arXiv:1401.7417
  [math.AG]}}, 2014.

\bibitem{Coates:cone}
Tom Coates.
\newblock The {Q}uantum {L}efschetz {P}rinciple for {V}ector {B}undles as a
  {M}ap {B}etween {G}ivental {C}ones.
\newblock \href{http://arxiv.org/abs/1405.2893}{\texttt{arXiv:1405.2893
  [math.AG] }}, 2014.

\bibitem{ccit}
Tom Coates, Alessio Corti, Hiroshi Iritani, and Hsian-Hua Tseng.
\newblock Computing genus-zero twisted {G}romov-{W}itten invariants.
\newblock {\em Duke Math. J.}, 147(3):377--438, 2009.

\bibitem{ccit2}
Tom Coates, Alessio Corti, Hiroshi Iritani, and Hsian-Hua Tseng.
\newblock A mirror theorem for toric stacks.
\newblock \href{http://arxiv.org/abs/1310.4163}{\texttt{arXiv:1310.4163
  [math.AG]}}, 2013.

\bibitem{qlhp}
Tom Coates, Amin Gholampour, Hiroshi Iritani, Yunfeng Jiang, Paul Johnson, and
  Cristina Manolache.
\newblock The quantum {L}efschetz hyperplane principle can fail for positive
  orbifold hypersurfaces.
\newblock {\em Math. Res. Lett.}, 19(5):997--1005, 2012.

\bibitem{cg}
Tom Coates and Alexander Givental.
\newblock Quantum {R}iemann-{R}och, {L}efschetz and {S}erre.
\newblock {\em Ann. of Math. (2)}, 165(1):15--53, 2007.

\bibitem{CIJ}
Tom Coates, Hiroshi Iritani, and Yunfeng Jiang.
\newblock The {C}repant {T}ransformation {C}onjecture for toric complete
  intersections.
\newblock \href{http://arxiv.org/abs/1410.0024}{\texttt{arXiv:1410.0024
  [math.AG]}}, 2014.

\bibitem{cclt}
Tom Coates, Yuan-Pin Lee, Alessio Corti, and Hsian-Hua Tseng.
\newblock The quantum orbifold cohomology of weighted projective spaces.
\newblock {\em Acta Math.}, 202(2):139--193, 2009.

\bibitem{FMN}
Barbara Fantechi, Etienne Mann, and Fabio Nironi.
\newblock Smooth toric {D}eligne-{M}umford stacks.
\newblock {\em J. Reine Angew. Math.}, 648:201--244, 2010.

\bibitem{GT}
Amin Gholampour and Hsian-Hua Tseng.
\newblock On computations of genus 0 two-point descendant {G}romov-{W}itten
  invariants.
\newblock {\em Michigan Math. J.}, 62(4):753--768, 2013.

\bibitem{Givental:ICM}
Alexander~B. Givental.
\newblock Homological geometry and mirror symmetry.
\newblock In {\em Proceedings of the {I}nternational {C}ongress of
  {M}athematicians, {V}ol.\ 1, 2 ({Z}\"urich, 1994)}, pages 472--480.
  Birkh\"auser, Basel, 1995.

\bibitem{Givental:toric}
Alexander~B. Givental.
\newblock A mirror theorem for toric complete intersections.
\newblock In {\em Topological field theory, primitive forms and related topics
  ({K}yoto, 1996)}, volume 160 of {\em Progr. Math.}, pages 141--175.
  Birkh\"auser Boston, Boston, MA, 1998.

\bibitem{gi}
Alexander~B. Givental.
\newblock Gromov-{W}itten invariants and quantization of quadratic
  {H}amiltonians.
\newblock {\em Mosc. Math. J.}, 1(4):551--568, 645, 2001.
\newblock Dedicated to the memory of I. G. Petrovskii on the occasion of his
  100th anniversary.

\bibitem{gi2}
Alexander~B. Givental.
\newblock Symplectic geometry of {F}robenius structures.
\newblock In {\em Frobenius manifolds}, Aspects Math., E36, pages 91--112.
  Friedr. Vieweg, Wiesbaden, 2004.

\bibitem{gw1}
Eduardo Gonzalez and Chris~T. Woodward.
\newblock Quantum cohomology and toric minimal model programs.
\newblock \href{http://arxiv.org/abs/1207.3253}{\texttt{arXiv:1207.3253
  [math.AG]}}, 2012.

\bibitem{Gross}
Mark Gross.
\newblock Mirror symmetry for {$\Bbb P^2$} and tropical geometry.
\newblock {\em Adv. Math.}, 224(1):169--245, 2010.

\bibitem{Iritani:QDM}
Hiroshi Iritani.
\newblock Quantum {$D$}-modules and generalized mirror transformations.
\newblock {\em Topology}, 47(4):225--276, 2008.

\bibitem{Iritani:periods}
Hiroshi Iritani.
\newblock Quantum cohomology and periods.
\newblock {\em Ann. Inst. Fourier (Grenoble)}, 61(7):2909--2958, 2011.

\bibitem{Iwanari1}
Isamu Iwanari.
\newblock The category of toric stacks.
\newblock {\em Compos. Math.}, 145(3):718--746, 2009.

\bibitem{Iwanari2}
Isamu Iwanari.
\newblock Logarithmic geometry, minimal free resolutions and toric algebraic
  stacks.
\newblock {\em Publ. Res. Inst. Math. Sci.}, 45(4):1095--1140, 2009.

\bibitem{Jiang}
Yunfeng Jiang.
\newblock The orbifold cohomology ring of simplicial toric stack bundles.
\newblock {\em Illinois J. Math.}, 52(2):493--514, 2008.

\bibitem{JT}
Yunfeng Jiang and Hsian-Hua Tseng.
\newblock Note on orbifold {C}how ring of semi-projective toric
  {D}eligne-{M}umford stacks.
\newblock {\em Comm. Anal. Geom.}, 16(1):231--250, 2008.

\bibitem{Jinzenji-Shimizu}
Masao Jinzenji and Masahide Shimizu.
\newblock Multi-point virtual structure constants and mirror computation of
  {$CP^2$}-model.
\newblock {\em Commun. Number Theory Phys.}, 7(3):411--468, 2013.

\bibitem{kkp}
Bumsig Kim, Andrew Kresch, and Tony Pantev.
\newblock Functoriality in intersection theory and a conjecture of {C}ox,
  {K}atz, and {L}ee.
\newblock {\em J. Pure Appl. Algebra}, 179(1-2):127--136, 2003.

\bibitem{ccliu}
Chiu-Chu~Melissa Liu.
\newblock Localization in {G}romov-{W}itten theory and orbifold
  {G}romov-{W}itten theory.
\newblock In {\em Handbook of moduli. {V}ol. {II}}, volume~25 of {\em Adv.
  Lect. Math. (ALM)}, pages 353--425. Int. Press, Somerville, MA, 2013.

\bibitem{Pandharipande:afterGivental}
Rahul Pandharipande.
\newblock Rational curves on hypersurfaces (after {A}. {G}ivental).
\newblock {\em Ast\'erisque}, (252):Exp.\ No.\ 848, 5, 307--340, 1998.
\newblock S{\'e}minaire Bourbaki. Vol. 1997/98.

\bibitem{Rose}
Michael~A. Rose.
\newblock A reconstruction theorem for genus zero {G}romov-{W}itten invariants
  of stacks.
\newblock {\em Amer. J. Math.}, 130(5):1427--1443, 2008.

\bibitem{ts}
Hsian-Hua Tseng.
\newblock Orbifold quantum {R}iemann-{R}och, {L}efschetz and {S}erre.
\newblock {\em Geom. Topol.}, 14(1):1--81, 2010.

\bibitem{Woodward:1}
Chris~T. Woodward.
\newblock Quantum {K}irwan morphism and {G}romov-{W}itten invariants of
  quotients {I}.
\newblock \href{http://arxiv.org/abs/1204.1765}{\texttt{arXiv:1204.1765
  [math.AG]}}, 2012.

\bibitem{Woodward:2}
Chris~T. Woodward.
\newblock Quantum {K}irwan morphism and {G}romov-{W}itten invariants of
  quotients {II}.
\newblock \href{http://arxiv.org/abs/1408.5864}{\texttt{arXiv:1408.5864
  [math.AG]}}, 2014.

\bibitem{Woodward:3}
Chris~T. Woodward.
\newblock Quantum {K}irwan morphism and {G}romov-{W}itten invariants of
  quotients {III}.
\newblock \href{http://arxiv.org/abs/1408.5869}{\texttt{arXiv:1408.5869
  [math.AG]}}, 2014.

\end{thebibliography}

\end{document}